\documentclass{amsart}

\usepackage{amsfonts}
\usepackage{amsmath}
\usepackage{amssymb}
\usepackage{amsthm}
\usepackage{mathrsfs}
\usepackage{bbm}
\usepackage{tikz}
\usetikzlibrary{decorations.pathmorphing,cd,patterns,decorations.pathreplacing}
\usepackage[colorlinks, citecolor = green!70!black, urlcolor = blue]{hyperref}
\usepackage[capitalise]{cleveref}

\DeclareFontFamily{U}{mathx}{\hyphenchar\font45}
\DeclareFontShape{U}{mathx}{m}{n}{
	<5> <6> <7> <8> <9> <10>
	<10.95> <12> <14.4> <17.28> <20.74> <24.88>
	mathx10
}{}
\DeclareSymbolFont{mathx}{U}{mathx}{m}{n}
\DeclareFontSubstitution{U}{mathx}{m}{n}
\DeclareMathAccent{\widecheck}{0}{mathx}{"71}
\DeclareMathAccent{\widebar}{0}{mathx}{"73}

\mathchardef\mhyphen="2D

\newtheorem{theorem}{Theorem}[section]
\newtheorem*{theorem*}{Theorem}
\newtheorem{proposition}[theorem]{Proposition}
\newtheorem{lemma}[theorem]{Lemma}

\newtheorem{claim}{Claim}[theorem]
\crefname{claim}{Claim}{Claims}

\theoremstyle{definition}
\newtheorem{definition}[theorem]{Definition}
\newtheorem*{notation}{Notation}

\theoremstyle{remark}
\newtheorem{remark}[theorem]{Remark}
\newtheorem{example}[theorem]{Example}
\crefname{example}{Example}{Examples}
\newtheorem{case}{Case}[claim]

\newcommand{\NN}{\mathbb{N}}
\newcommand{\ZZ}{\mathbb{Z}}

\renewcommand{\O}{\mathcal{O}}

\newcommand{\C}{\mathscr{C}}

\newcommand{\im}{\operatorname{im}}
\newcommand{\defeq}{\overset{\text{def}}{=}}
\newcommand{\op}{^\mathrm{op}}
\newcommand{\id}{\mathrm{id}}
\newcommand{\incl}{\hookrightarrow}

\newcommand{\ds}{\displaystyle}

\newcommand{\omegaCat}{\omega\mhyphen\underline{\mathrm{Cat}}}
\newcommand{\Set}{\underline{\mathrm{Set}}}

\newcommand{\sSet}{\underline{\mathrm{sSet}}}
\newcommand{\msSet}{\underline{\mathrm{msSet}}}
\newcommand{\ADC}{\underline{\mathrm{ADC}}}

\newcommand{\filler}[1]{\mathbin{\bullet_{#1}}}
\newcommand{\paste}[1]{\mathbin{\circ_{#1}}}
\newcommand{\face}[1]{\delta_{#1}}
\newcommand{\degen}[1]{\sigma_{#1}}
\newcommand{\horn}[2]{\Lambda^{#2}[#1]}
\newcommand{\asimp}[2]{\Delta^{#2}[#1]}
\newcommand{\asimpd}[2]{\Delta^{#2}[#1]'}
\newcommand{\asimpdd}[2]{\Delta^{#2}[#1]''}

\newcommand{\rank}{\operatorname{rank}}
\newcommand{\corank}{\operatorname{corank}}
\newcommand{\level}{\operatorname{level}}
\newcommand{\supp}{\operatorname{supp}}
\newcommand{\bbar}[1]{\bar{\bar{#1}}}

\newcommand{\Deltainj}{\Delta_{\mathrm{inj}}}
\newcommand{\join}{\oplus}
\newcommand{\quadand}{\quad\text{and}\quad}
\newcommand{\lex}{_{\mathrm{lex}}}

\colorlet{palegray}{gray!50!white}

\title{Orientals as free weak $\omega$-categories}
\author{Yuki Maehara}

\address{Institute of Mathematics for Industry, Kyushu University, Fukuoka, Japan}
\email{y-maehara@imi.kyushu-u.ac.jp}

\makeatletter
\@namedef{subjclassname@2020}{%
	\textup{2020} Mathematics Subject Classification}
\makeatother


\begin{document}
	\maketitle


	\begin{abstract}
		The \emph{orientals} are the free strict $\omega$-categories on the simplices introduced by Street.
		The aim of this paper is to show that they are also the free \emph{weak} $\omega$-categories on the same generating data.
		More precisely, we exhibit the complicial nerves of the orientals as fibrant replacements of the simplices in Verity's model structure for weak complicial sets.
	\end{abstract}

	\section*{Introduction}
	This paper concerns the simplicial approach to the theory of higher-dimensional categories.
	The fundamental notion in this approach is that of \emph{complicial set} introduced by Roberts \cite{Roberts:complicial}; a simplicial set equipped with a distinguished subset of \emph{marked} simplices satisfying horn-filling conditions akin to those for Kan complexes.
	Although Roberts conjectured the complicial sets to be precisely the nerves of the (strict) $\omega$-categories, even constructing the desired nerve functor proved to be rather difficult.
	The construction was eventually achieved by Street \cite{Street:oriented} who realised the simplices as $\omega$-categories in the form of \emph{orientals}.
	This led to a precise formulation of Roberts' conjecture, which was subsequently proven by Verity \cite[Theorem 266]{Verity:strict}.
	
	Verity then began the study of \emph{weak complicial sets} \cite{Verity:I,Verity:II}, a homotopical variant of the strict notion.
	They model \emph{weak $\omega$-categories} (or \emph{$(\infty,\infty)$-categories}), \emph{i.e.}\;structures that only satisfy the usual axioms for $\omega$-category up to coherent, invertible higher cells.
	Among many other results, Verity constructed a model structure capturing the homotopy theory of weak complicial sets.
	
	The aim of this paper is to contribute to this theory by showing that the orientals, originally introduced as the free \emph{strict} $\omega$-categories on the simplices, are also the free \emph{weak} $\omega$-categories on the same generating data.
	More precisely, in \cref{main} we exhibit the complicial nerve of the $n$-th oriental $\O_n$ as a fibrant replacement of $\Delta[n]$ in Verity's aforementioned model structure.
	(Consequently, the nerve of $\O_n$ is weakly equivalent to $\Delta[n]$ in any Bousfield localisation thereof such as the saturated and $n$-trivial complicial model structures.)
	
	In many ways, this paper draws insights from Steiner's analysis of the orientals in \cite{Steiner:orientals}.
	In particular, we make use of his description of the $\omega$-functors $\O_m \to \O_n$ as certain formal linear combinations of maps $[m] \to [n]$ in $\Delta$, and much of our proof is based on ideas that can be found in \cite[Propositions 5.9-11]{Steiner:orientals}.
	
	After posting the first version of this paper on arXiv, we learnt that Gagna, Ozornova and Rovelli were working on a similar problem.
	Their preprint is now available at \cite{GOR:cones}.
	Since they consider cones over (or under) more general $\omega$-categories than just the orientals, their result is strictly stronger than ours.
	On the other hand, if one is only interested in the special cases of the orientals, our combinatorics is considerably shorter.

	
	\section{Background}
	The combinatorics in this paper relies on Steiner's description \cite{Steiner:orientals} of the $\omega$-functors between the orientals as certain formal linear combinations of maps in $\Delta$.
	This description is reviewed in \cref{subsection O}, and \cref{subsection pasting} discusses how Steiner recovers a notion of composition in this framework.
	On the other hand, the main result of this paper is formulated using the language of complicial sets.
	We present minimal background on complicial sets in \cref{subsection nerve}, and refer the interested reader to \cite{Verity:strict,Verity:I,Verity:II} for more on this subject.
	\cref{subsection complicial} discusses how the notions of identity and composition in the two frameworks agree.

	\begin{notation}
		Given a simplicial set $X \in \sSet \defeq \bigl[\Delta^{\mathrm{op}},\Set\bigr]$, an element $x \in X_n$ and a simplicial operator $\alpha : [m] \to [n]$, we write $x\alpha$ for the image of $x$ under $X(\alpha) : X_n \to X_m$.
		Note that we have $(x\alpha)\beta = x(\alpha\beta)$ in this notation whenever either side is defined.
	\end{notation}

	\subsection{Orientals and Steiner's category $\O$}\label{subsection O}
	The \emph{$n$-th oriental} $\O_n$, introduced by Street \cite{Street:oriented}, is the free $\omega$-category on the $n$-simplex with its atomic $m$-cells in bijective correspondence with the injective maps $[m] \to [n]$ in $\Delta$.
	The following is a slightly more precise (but still informal) description of $\O_n$.
	\begin{itemize}
		\item The zeroth oriental $\O_0$ is the terminal $\omega$-category, consisting of a single $0$-cell and no non-identity higher cells.
		\item For $n \ge 1$, the ``boundary'' of $\O_n$ may be constructed by gluing $n+1$ copies of $\O_{n-1}$ according to the structure of $\partial\Delta[n]$.
		The ``interior'' of $\O_n$ is then filled with an $n$-cell that points from the composite of odd faces to the composite of even ones.
	\end{itemize}
	We have drawn (the atomic cells in) $\O_n$ for $n \le 3$ in \cref{example0,example3}.
	These drawings are essentially taken from \cite{Street:oriented}, where $\O_4$, $\O_5$ and $\O_6$ can also be found.
	In these figures, and also throughout this paper, we denote a simplicial operator $\alpha : [m] \to [n]$ by the sequence of its images $\bigl(\alpha(0),\dots,\alpha(m)\bigr)$.
	
	\begin{figure}
		\[
		\begin{tabular}{c|c|c}
			$\O_0$ & $\O_1$ & $\O_2$ \\\hline
			\begin{tikzpicture}[baseline = 0]
				\node (0) at (0,0) {$(0)$};
			\end{tikzpicture}
			&
			\begin{tikzpicture}[baseline = 0]
				\node (0) at (0,0) {$(0)$};
				\node (1) at (2,0) {$(1)$};
				
				\draw[->] (0.east) -- (1.west);
				
				\node[scale = 0.8] at (1,0.2) {$(0,1)$};
			\end{tikzpicture}
			&
			\begin{tikzpicture}[baseline = 30]
				\node (0) at (0,0) {$(0)$};
				\node (2) at (4,0) {$(2)$};
				\node (1) at (2,2) {$(1)$};
				
				\draw[->] (0.east) -- (2.west);
				\draw[->] (0.north east) -- (1.south west);
				\draw[->] (1.south east) -- (2.north west);
				
				\draw[->, double] (2,0.2) -- (2,1.5);
				
				\node[scale = 0.8] at (2,-0.2) {$(0,2)$};
				\node[scale = 0.8] at (0.6,1.1) {$(0,1)$};
				\node[scale = 0.8] at (3.4,1.1) {$(1,2)$};
				
				\node[scale = 0.8, fill = white] at (2,0.8) {$(0,1,2)$};
			\end{tikzpicture}
		\end{tabular}
		\]
		\caption{$\O_0$, $\O_1$ and $\O_2$}\label{example0}
	\end{figure}
	
	\begin{figure}
		\[
		\begin{tikzpicture}[baseline = 30]
			\node (0) at (0,0) {$(0)$};
			\node (1) at (1,2) {$(1)$};
			\node (2) at (3,2) {$(2)$};
			\node (3) at (4,0) {$(3)$};
			
			\draw[->] (0.4,0) -- (3.6,0);
			\draw[->] (0.2,0.4) -- (0.8,1.6);
			\draw[->] (1.4,2) -- (2.6,2);
			\draw[->] (3.2,1.6) -- (3.8,0.4);
			
			\draw[->] (0.3,0.2) -- (2.7,1.8);
			
			\node[scale = 0.8] at (0.2,1.1) {$(0,1)$};
			\node[scale = 0.8] at (3.8,1.1) {$(2,3)$};
			\node[scale = 0.8] at (2,-0.2) {$(0,3)$};
			\node[scale = 0.8] at (2,2.2) {$(1,2)$};
			\node[scale = 0.8] at (1.8,0.8) {$(0,2)$};
			
			\draw[->, double] (1.2,1) -- (1.2,1.7);
			\node[scale = 0.6, fill = white] at (1.2,1.3) {$(0,1,2)$};
			
			\draw[->, double] (2.8,0.2) -- (2.8,1.6);
			\node[scale = 0.6, fill = white] at (2.8,0.8) {$(0,2,3)$};
		\end{tikzpicture}
		\begin{tikzpicture}[baseline = 0]
			\draw (0,0) -- (2,0)
			(0,0.05) -- (1.95,0.05)
			(0,-0.05) -- (1.95,-0.05)
			(2,0) + (-0.2,0.2) --+ (0,0) --+ (-0.2,-0.2);
			\node[scale = 0.8] at (1,0.3) {$(0,1,2,3)$};
		\end{tikzpicture}
		\begin{tikzpicture}[baseline = 30]
			\node (0) at (0,0) {$(0)$};
			\node (1) at (1,2) {$(1)$};
			\node (2) at (3,2) {$(2)$};
			\node (3) at (4,0) {$(3)$};
			
			\draw[->] (0.4,0) -- (3.6,0);
			\draw[->] (0.2,0.4) -- (0.8,1.6);
			\draw[->] (1.4,2) -- (2.6,2);
			\draw[->] (3.2,1.6) -- (3.8,0.4);
			
			\draw[->] (1.3,1.8) -- (3.7,0.2);
			
			\node[scale = 0.8] at (0.2,1.1) {$(0,1)$};
			\node[scale = 0.8] at (3.8,1.1) {$(2,3)$};
			\node[scale = 0.8] at (2,-0.2) {$(0,3)$};
			\node[scale = 0.8] at (2,2.2) {$(1,2)$};
			\node[scale = 0.8] at (2.2,0.8) {$(1,3)$};
			
			\draw[->, double] (2.8,1) -- (2.8,1.7);
			\node[scale = 0.6, fill = white] at (2.8,1.3) {$(1,2,3)$};
			
			\draw[->, double] (1.2,0.2) -- (1.2,1.6);
			\node[scale = 0.6, fill = white] at (1.2,0.8) {$(0,1,3)$};
		\end{tikzpicture}
		\]
		\caption{$\O_3$}\label{example3}
	\end{figure}
	
	\begin{remark}
	Recall that we may regard a small category as a single set $\C$ (of morphisms) equipped with source and target maps $s,t : \C \to \C$ and a composition map
	\[
	* : \bigl\{(f,g) \in \C \times \C : t(f) = s(g)\bigr\} \to \C
	\]
	satisfying suitable axioms; the objects of $\C$ are recovered as the fixed points of $s$ (or those of $t$, which coincide thanks to one of the axioms).
	Extending this view, we regard a small $\omega$-category as a set $\C$ equipped with $s_m,t_m : \C \to \C$ and
	\[
	*_m : \bigl\{(f,g) \in \C \times \C : t_m(f) = s_m(g)\bigr\} \to \C
	\]
	for each $m \ge 0$, satisfying suitable axioms.
	So an $m$-cell $a$ in $\C$ is a fixed point of $s_m$ (or equivalently of $t_m$), and the identity $(m+1)$-cell $\id_a$ is $a$ itself.
	See \cite[Definition 2.1]{Steiner:omega} (which is equivalent to \cite[\textsection 1]{Street:oriented} plus an extra finite-dimensionality axiom) for the precise definition.
	We write $\omegaCat$ for the category of small $\omega$-categories and $\omega$-functors.
	\end{remark}
	
	In principle, one may work with the original construction of $\O_n$ from \cite{Street:oriented}.
	That is, an $m$-cell $a$ in $\O_n$ is a suitably compatible (in the sense of \cite[p.313(2)]{Street:oriented}) family
	\[
	a = \begin{pmatrix}
		a^1_0 & \dots & a^1_{m-1} & a^1_m = a^0_m & \varnothing & \dots \\
		a^0_0 & \dots & a^0_{m-1} & a^1_m = a^0_m & \varnothing & \dots
	\end{pmatrix}
	\]
	where each $a^\epsilon_i$ is a \emph{well-formed} (in the sense of \cite[p.311(WF)]{Street:oriented}) set of injective maps $[i] \to [n]$ in $\Delta$.
	Here $a^1_i$ (respectively $a^0_i$) encodes the atomic decomposition of the $i$-source (respectively the $i$-target) of $a$, and the well-formedness is a combinatorial condition assuring that the elements of each $a^\epsilon_i$ indeed compose to a single $i$-cell.
	However, it is rather difficult to describe $\omega$-functors between them in this way.
	Instead, we now recall Steiner's description \cite{Steiner:orientals} of the full subcategory of $\omegaCat$ spanned by the orientals, which is much more tractable.
	
	Consider the free abelian group-enriched category $\ZZ\Delta$ on $\Delta$.
	More explicitly, $\ZZ\Delta$ has the same objects as $\Delta$ and its hom-abelian groups $\ZZ\Delta\bigl([m],[n]\bigr)$ are the free ones generated by the corresponding hom-sets of $\Delta$.
	The composition in $\ZZ\Delta$ is given by extending the one in $\Delta$ linearly in each variable.
	We will write $x_\alpha$ for the coefficient of $\alpha : [m] \to [n]$ in $x \in \ZZ\Delta\bigl([m],[n]\bigr)$ so that
	\[
	x = \sum_{\alpha : [m] \to [n]}x_\alpha\cdot\alpha.
	\]
	For given $x \in \ZZ\Delta\bigl([m],[n]\bigr)$, we define its \emph{support} by
	\[
	\supp(x) \defeq \bigl\{\alpha \in \Delta\bigl([m],[n]\bigr) : x_\alpha \neq 0\bigr\}.
	\]
	
	\begin{definition}\label{O definition}
		Let $\O(m,n)$ denote the subset of $\ZZ\Delta\bigl([m],[n]\bigr)$ consisting of those $x$ such that:
		\begin{itemize}
			\item[(O1)] $\sum_\alpha x_\alpha = 1$; and
			\item[(O2)] for any injective maps $\beta : [p] \to [m]$ and $\gamma : [p] \to [n]$ in $\Delta$, the coefficient $(x\beta)_\gamma$ is non-negative.
		\end{itemize}
	Here $x\beta$ denotes the composite in $\ZZ\Delta$ where $\beta$ is identified with its image under the canonical functor $\Delta \to \ZZ\Delta$.
	\end{definition}
	The ultimate (but unenlightening) justification for the rather mysterious condition (O2) is \cref{O} below.
	A more conceptual justification is given in \cref{rmk}.

\begin{theorem}[{\cite[Theorem 4.2]{Steiner:orientals}}]\label{O}
	The sets $\O(m,n)$ determine a subcategory $\O$ of (the underlying ordinary category of) $\ZZ\Delta$.
	Moreover, this category $\O$ is isomorphic to the full subcategory of $\omegaCat$ spanned by the orientals.
\end{theorem}
\begin{remark}\label{intuition}
	Intuitively, the $\omega$-functor $\O_m \to \O_n$ corresponding to $x \in \O(m,n)$ is obtained by pasting the positive terms in $x$ (regarded as cells in $\O_n$) along the negative ones.
	For example,
	\[
	(0,1)-(1,1)+(1,2) \in \O(1,2)
	\]
	corresponds to the $\omega$-functor $\O_1 \to \O_2$ that picks out the composite of the $1$-cells $(0,1)$ and $(1,2)$ (\cref{Omn} left).
	
	Since we are mapping out of simplices rather than globes, there is some flexibility to the action of the $\omega$-functor on the ``boundary'' even after specifying its action on the ``interior''.
	This is controlled by positive, degenerate terms in $x \in \O(m,n)$.
	For example,
	\[
	(0,1,2)-(1,1,2)+(1,2,2) \in \O(2,2)
	\]
	corresponds to the unique $\omega$-functor $\O_2 \to \O_2$ that sends the atomic $1$-cells $(0,1)$, $(1,2)$, and $(0,2)$ to the composite $(0,1)-(1,1)+(1,2)$, the identity at $(2)$, and $(0,2)$ respectively  (\cref{Omn} middle).
	The non-degenerate term $(0,1,2)$ represents the identity $\omega$-functor on $\O_2$, but pasting it with $(1,2,2)$ induces the non-trivial action on the boundary.
	
	Positive degenerate terms can also represent whiskering.
	For example,
	\[
	(0,1,2)-(2,2,2)+(2,2,3) \in \O(2,3)
	\]
	represents the $3$rd face of $\O_3$ whiskered with the $1$-cell $(2,3)$  (\cref{Omn} right).
\end{remark}

\begin{figure}
	\[
	\begin{tikzpicture}
		\filldraw
		(0,0) circle [radius = 1pt]
		(2,0) circle [radius = 1pt];
		
		\draw[->] (0.1,0.1) -- (1,1) -- (1.9,0.1);
		
		\draw[->, palegray] (0.15,0) -- (1.85,0);
		\draw[->, double, palegray] (1,0.2) -- (1,0.6);
	\end{tikzpicture}
	\quad\quad\quad
	\begin{tikzpicture}
		\filldraw
		(0,0) circle [radius = 1pt]
		(2,0) circle [radius = 1pt]
		(1.7,0.3) circle [radius = 1pt];
		
		\draw[->] (0.1,0.1) -- (1,1) -- (1.6,0.4);
		\draw[->] (0.15,0) -- (1.85,0);
		\draw[double] (1.8,0.2) -- (1.9,0.1);
		\draw[->, double] (1,0.2) -- (1,0.6);
	\end{tikzpicture}
	\quad\quad\quad
	\begin{tikzpicture}
		\filldraw
		(0,0) circle [radius = 1pt]
		(0.5,1) circle [radius = 1pt]
		(2,0) circle [radius = 1pt];
		
		\draw[->, palegray] (0.15,0) -- (1.85,0);
		\draw[->] (0.05,0.1) -- (0.45,0.9);
		\draw[->] (0.6,1) -- (1.55,1) -- (2,0.1);
		\draw[->] (0.15,0.1) -- (1.5,0.95) -- (1.95,0.05);
		
		\draw[->, double] (0.65,0.55) -- (0.65,0.85);
		\draw[->, double, palegray] (1.35,0.2) -- (1.35,0.6);
	\end{tikzpicture}
	\]
	\caption{$\omega$-functors between orientals}\label{Omn}
\end{figure}

\begin{remark}\label{rmk}	
	We now give a justification for the conditions (O1) and (O2).
	According to \cref{intuition}, the sum $\sum_\alpha x_\alpha$ is the difference between:
	\begin{itemize}
		\item the number of (generating) cells we are pasting; and
		\item the number of cells \emph{along which} we are pasting the above cells.
	\end{itemize}
	Since we are pasting (topologically) contractible cells and we wish to obtain a contractible cell, this difference must be $1$.
	If it is greater, the result would be disconnected; if it is smaller, we must paste something to itself at some point, resulting in a non-trivial loop.
	This is our intuition behind (O1).
	
	For (O2), let us first consider the special case where $\beta$ is the identity at $[m]$.
	Then (O2) requires all negative terms in $x$ be degenerate.
	According to \cref{intuition}, this is just the reasonable statement that we only paste $m$-cells along cells of dimension strictly smaller than $m$.
	The general case (with arbitrary injective $\beta$) asks the same condition to hold not just for $x$ but also for its faces.
\end{remark}

The canonical functor $\Delta \to \ZZ\Delta$ can be easily checked to factor through $\O$, and the composite $\Delta \to \O \incl \omegaCat$ is precisely the cosimplicial object Street used to define the nerve functor.
(See \cref{subsection nerve} below.)
In particular, we obtain the simplicial nerve of $\O_n$ as follows.

\begin{definition}
	We will write $\O(-,n)$ for the simplicial set whose $m$-simplices are precisely the elements of $\O(m,n)$.
	The simplicial action is defined by restricting the obvious representable $\O\op \to \Set$ along $\Delta\op \to \O\op$.
	The action of each $\beta : [p] \to [m]$ is thus given by linearly extending the precomposition
	\[
	\Delta\bigl([m],[n]\bigr) \to \Delta\bigl([p],[n]\bigr) : \alpha \mapsto \alpha\beta.
	\]
\end{definition}

Note that $x \in \O(m,n)$ is degenerate at some $k$ if and only if each $\alpha \in \supp(x)$ is degenerate at $k$.

\begin{remark}
	It is easy to deduce, either directly from the definition of $\O$ or from \cite[Proof of Proposition 5.7]{Steiner:orientals}, that we have
	\[
	\O(0,n) = \bigl\{(i) : i \in [n]\bigr\}
	\]
	for any $n \ge 0$.
	This provides an obvious bijection $[n] \cong \O(0,n)$, and we will often identify the two sets accordingly.
	In particular, given $x \in \O(m,n)$ and $i \in [m]$, we will treat $x(i)$ (the image of $x$ under the action of $(i):[0] \to [m]$) as if it were a natural number.
\end{remark}

\begin{proposition}[{\cite[Proposition 5.7]{Steiner:orientals}}]\label{terminus}
Let $x \in \O(m,n)$.
Then we have
\[
x(0) = \min\{\alpha(0) : \alpha \in \supp(x)\}
\]
and
\[
x(m) = \max\{\alpha(m) : \alpha \in \supp(x)\}.
\]
\end{proposition}

\subsection{The operations $\paste k$ and $\filler k$}\label{subsection pasting}
Now we recall two families of operations on the simplicial set $\ZZ\Delta\bigl(-,[n]\bigr)$ from \cite{Steiner:orientals} which may be interpreted as a sort of \emph{pasting} and its \emph{witness} when restricted to $\O(-,n)$.

We denote the $i$-th elementary face (respectively degeneracy) operator by $\face i$ (respectively $\degen i$).

\begin{definition}\label{filler definition}
	Let $x,y \in \ZZ\Delta\bigl([m],[n]\bigr)$ with $m \ge 1$, and let $1 \le k \le m$.
	Suppose that $x\face{k-1} = y\face{k}$ holds.
	Then we define
	\begin{align*}
		x \filler k y &\defeq  x\degen{k}-x\face{k-1}\degen{k-1}^2+y\degen{k-1},\\
		x \paste k y & \defeq x-x\face{k-1}\degen{k-1}+y
	\end{align*}
	or equivalently,
	\begin{align*}
	x \filler k y &\defeq  x\degen{k}-y\face{k}\degen{k-1}^2+y\degen{k-1},\\
	x \paste k y &\defeq x-y\face{k}\degen{k-1}+y.
	\end{align*}
\end{definition}

\begin{remark}
	The assumption $x\face{k-1} = y\face k$ in \cref{filler definition} should be interpreted as a composability condition.
	Then according to \cref{intuition}, if $x,y \in \O(m,n)$ then $x \paste k y$ is precisely the pasting of $x$ and $y$ along their common face.
	The simplex $x \filler k y$ is to be thought of as a witness for this pasting (see \cref{filler basic} below).
	See also \cite[Remark 6.4]{Steiner:universal} for how $\paste k$ can indeed be seen as a composition in a suitable category.
\end{remark}

\begin{remark}
	The operations we are denoting by $\paste k$ and $\filler k$ are what would be called $\vee_{k-1}$ and $\triangledown_{k-1}$ respectively in \cite{Steiner:orientals}.
	We have shifted the index because, while working with the combinatorics of $\O(-,n)$, we found it much less confusing for the $k$-th face (rather than the $(k+1)$-st one) of an object labelled with $k$ to play a special role.
	We apologise to the reader if they are already familiar with Steiner's work and find our notation confusing.
	In Steiner's later papers such as \cite{Steiner:universal}, the symbol $\triangledown$ is replaced by $\wedge$ and the operations corresponding to $\paste k$ go unnamed.
\end{remark}

The following proposition is a direct consequence of \cref{filler definition}.

\begin{proposition}\label{filler basic}
	For any $x,y,k$ as in \cref{filler definition}, we have
	\[
	(x \filler k y)\face{k-1} = y, \quad (x \filler k y)\face{k} = x \paste k y, \quadand (x \filler k y)\face{k+1} = x.
	\]
\end{proposition}

\begin{proposition}[{\cite[Proposition 5.4]{Steiner:orientals}}]\label{O pasting}
	Let $x,y \in \O(m,n)$ with $m \ge 1$, and let $1 \le k \le m$.
	Suppose $x\face{k-1} = y\face{k}$.
	Then we have
	\[
	x \filler k y \in \O(m+1,n) \quadand x \paste k y \in \O(m,n).
	\]
\end{proposition}

\begin{proposition}\label{filler}
	An element $x \in \ZZ\Delta\bigl([m],[n]\bigr)$ satisfies $x = (x\face{k+1}) \filler k (x\face{k-1})$ if and only if each simplicial operator in $\supp(x)$ is degenerate at either $k-1$ or $k$.
\end{proposition}
\begin{proof}
	The ``only if'' direction is clear.
	For the ``if'' direction, suppose that each simplicial operator in $\supp(x)$ is degenerate at either $k-1$ or $k$.
	Since we have
	\[
	(x\face{k+1})\filler k (x\face{k-1}) = x\face{k+1}\degen k-x\face{k+1}\face{k-1}\degen{k-1}^2+x\face{k-1}\degen{k-1},
	\]
	it suffices to check that
	\[
	\alpha = \alpha\face{k+1}\degen k-\alpha\face{k+1}\face{k-1}\degen{k-1}^2+\alpha\face{k-1}\degen{k-1}
	\]
	holds for any $\alpha : [m] \to [n]$ that is degenerate at either $k-1$ or $k$.
	The latter statement is straightforward to verify.
\end{proof}

The following proposition turns a sum ($x=y+z$) into a pasting ($x=u\paste{k}v$), and particular instances of this fact are used in \cite[Propositions 5.9-11]{Steiner:orientals}.
Its proof is a straightforward manipulation of simplicial operators.

\begin{proposition}\label{decomposition}
	Let $x,y,z \in \ZZ\Delta\bigl([m],[n]\bigr)$ be elements satisfying $x=y+z$, and let $1 \le k \le m$.
	Define
	\[
	u = y+z\face k \degen{k-1}, \quadand v = y\face{k-1}\degen{k-1} + z.
	\]
	Then we have
	\[
	u\face{k-1} = v \face{k}, \quad u \filler k v = y\degen{k} + z\degen{k-1}, \quadand u \paste k v = x.
	\]
\end{proposition}

\begin{remark}
	Observe that most of the operations we consider only preserve the support in a weak sense.
	For instance, we have
	\begin{align*}
		\supp(x \face k) &\subset \bigl\{\alpha\face k : \alpha \in \supp(x)\bigr\},\\
		\supp(x \paste k y) &\subset \supp(x) \cup \bigl\{\alpha\face{k-1}\degen{k-1} : \alpha \in \supp(x)\bigr\} \cup \supp(y)
	\end{align*}
	for $x,y \in \ZZ\Delta\bigl([m],[n]\bigr)$ whenever $x\face k$ or $x \paste k y$ makes sense.
	In general, one cannot replace the subset symbol by an equality; one has, for example,
	\[
	\bigl((0,1)-(1,1)+(1,2)\bigr)\face 0 = (2)
	\]
	in $\O(-,2)$ and
	\[
	\bigl((0,0,2)-(0,2,2)+(0,2,3)\bigr) \paste 1 \bigl((0,1,2)-(2,2,2)+(2,2,3)\bigr) = (0,1,2)-(0,2,2)+(0,2,3)
	\]
	in $\O(-,3)$.
	Nevertheless, the support \emph{is} strictly preserved in certain nice cases, and this will be crucial in our arguments below.
\end{remark}

\subsection{Complicial sets}\label{subsection nerve}
The cosimplicial object $\Delta \to \omegaCat : [n] \mapsto \O_n$ as described in the previous subsection induces a nerve functor $\omegaCat \to \sSet$.
However, this functor is not full; for instance, the standard $2$-simplex is the nerve of the obvious $1$-category $[2]$, and we can consider the (unit) simplicial map $\Delta[2] \to \O(-,2)$ picking out the simplex $(0,1,2)$.
This map does not come from an $\omega$-functor $[2] \to \O_2$ since we are sending the commutative triangle in $[2]$ to a non-commutative one in $\O_2$.

To rectify this, Roberts proposed considering simplicial sets with distinguished simplices (to be thought of as ``abstract commutative/identity simplices'').
The adjective to refer to these distinguished simplices has changed multiple times from Roberts' original \emph{neutral}, to \emph{hollow}, to \emph{thin}, and then to \emph{marked}.

\begin{definition}
	A \emph{marked simplicial set} $(X,tX)$ is a simplicial set $X$ together with subsets $tX_n \subset X_n$ of \emph{marked} simplices for $n \ge 1$ containing all degenerate simplices.
	A \emph{morphism} of marked simplicial sets $f : (X,tX) \to (Y,tY)$ is a simplicial map $f : X \to Y$ that preserves marked simplices.
	We denote the category of marked simplicial sets by $\msSet$.
\end{definition}

We will often suppress $tX$ and simply speak of a marked simplicial set $X$.

\begin{remark}
	The reader is warned that this notion of marked simplicial set is different from Lurie's \cite[\textsection 3.1]{HTT} where marked simplices are allowed only in dimension $1$.
\end{remark}

Now we can give the precise definition of Street's nerve functor.

\begin{definition}
	The \emph{complicial nerve} of an $\omega$-category $\C$ is the marked simplicial set
	\[
	\omegaCat(\O_{(-)},\C)
	\]
	in which an $n$-simplex $F : \O_n \to \C$ is marked if and only if it sends the (unique) atomic $n$-cell in $\O_n$ to an $(n-1)$-cell in $\C$. 
\end{definition}

With this definition, Street showed in \cite{Street:fillers} that the nerve of any $\omega$-category is a \emph{complicial set} in the sense we recall now.
(In fact, Street proved something stronger; see \cite{Street:fillers} for the precise statement.)

\begin{definition}
	We say a morphism in $\msSet$ is \emph{regular} if it reflects marked simplices.
	In other words, $f : (X,tX) \to (Y,tY)$ is regular if $f(x) \in tY_n$ implies $x \in tX_n$.
	By a \emph{regular subset} $(A,tA)$ of $(X,tX)$, we mean a simplicial subset $A \subset X$ equipped with the marking $tA_n = A_n \cap tX_n$.
\end{definition}

\begin{definition}\label{anodyne}
	For each $n \ge 0$, we will regard the standard $n$-simplex $\Delta[n]$ as a marked simplicial set equipped with the minimal marking (\emph{i.e.}\;only the degenerate simplices are marked).
	Also, for each $0 < k < n$, we write:
	\begin{itemize}
		\item $\asimp n k$ for the object obtained from $\Delta[n]$ by further marking those simplices $\alpha : [m] \to [n]$ with $\{k-1, k, k+1\} \subset \im\alpha$;
		\item $\horn n k$ for the regular subset of $\asimp n k$ whose underlying simplicial set is the $k$-th horn;
		\item $\asimpd n k$ for the object obtained from $\asimp n k$ by further marking $\face{k-1}$ and $\face{k+1}$; and
		\item $\asimpdd n k$ for the object obtained from $\asimpd n k$ by further marking $\face{k}$.
	\end{itemize}
	The class of \emph{inner complicial anodyne extensions} is the closure of the union
	\[
	\bigl\{\horn n k \incl \asimp n k : 0 < k < n\bigr\} \cup \bigl\{\asimpd n k \incl \asimpdd n k : 0 < k < n\bigr\}
	\]
	under coproducts, pushouts along arbitrary maps, and transfinite compositions.
	(We are not closing this class under retracts.)
\end{definition}

\begin{definition}\label{complicial definition}
	A (\emph{strict}) \emph{complicial set} is a marked simplicial set $X$ such that:
	\begin{itemize}
		\item $X$ has the unique right lifting property with respect to $\horn n k \incl \asimp n k$ for all $0<k<n$;
		\item $X$ has the unique right lifting property with respect to $\asimpd n k \incl \asimpdd n k$ for all $0<k<n$; and
		\item the marked $1$-simplices in $X$ are precisely the degenerate ones.
	\end{itemize}
\end{definition}

Now we can state what was conjectured by Roberts, made precise by Street, and proven by Verity.

\begin{theorem}[{\cite[Theorem 266]{Verity:strict}}]\label{Dom}
	The complicial nerve provides an equivalence between $\omegaCat$ and the full subcategory of $\msSet$ spanned by the complicial sets.
\end{theorem}

\begin{remark}
	Since $k-1$ and $k+1$ have the same parity and $k$ has the opposite one, we see from the informal description of $\O_n$ that $\face{k-1}$ and $\face{k+1}$ lie on the same side (either source or target) of the interior $n$-cell, and $\face{k}$ lies on the other.
	Observe further that, aside from the degenerate ones, the marked simplices in $\asimp n k$ are precisely those that are not contained in $\face{k-1}$, $\face{k}$, or $\face{k+1}$.
	Thus the (marked) interior $n$-simplex of $\asimp n k$ may be thought of as asserting an equality between $\face{k}$ and the composite of $\face{k-1}$ and $\face{k+1}$.
	
	In this sense, lifting against $\horn n k \incl \asimp n k$ defines a sort of composition, and lifting against $\asimpd n k \incl \asimpdd n k$ ensures that any composite of identities is itself an identity.
\end{remark}

\begin{remark}
	Although the term \emph{complicial set} originally meant \cref{complicial definition}, today it is more commonly used to refer to the weak variant where:
	\begin{itemize}
		\item the unique right lifting property is replaced by a mere right lifting property;
		\item the lifting is required for $k=0$ and $k=n$ too; and
		\item the condition on marked $1$-simplices is dropped.
	\end{itemize}
	(The definitions of horns and simplices for $k=0$ and $k=n$ can be found at \cite[Notation 10]{Verity:I}.
	By \cite[Example 17]{Verity:I}, complicial sets in the sense of \cref{complicial definition} are complicial sets in this sense.)
	The homotopy theory of these \emph{weak complicial sets} is captured by a model structure on $\msSet$ due to Verity \cite[Theorem 100]{Verity:I}.
	In particular, the weak complicial sets are precisely the fibrant objects therein.
	Since the inner complicial anodyne extensions of \cref{anodyne} are examples of its trivial cofibrations, this model structure provides a suitable framework in which to interpret our main result (\cref{main}).
	Note, however, that this model structure will not play any mathematical role in this paper.
\end{remark}

\subsection{$\O(-,n)$ as complicial sets}\label{subsection complicial}

From now on, we will regard $\O(-,n)$ as the complicial nerve of $\O_n$ and in particular as a \emph{marked} simplicial set.
The following proposition states that the marked simplices in $\O(-,n)$ are precisely those that can be obtained by pasting degenerate simplices.

\begin{proposition}\label{marked}
	A simplex $x \in \O(m,n)$ is marked if and only if each simplicial operator in $\supp(x)$ is degenerate.
\end{proposition}
\begin{proof}
	We proceed by analysing explicitly the chain of bijections connecting the hom-set $\omegaCat(\O_m,\O_n)$ to $\O(m,n)$.
	
	Let $\ADC$ denote the category of \emph{augmented directed complexes}.
	That is, an object $K$ in $\ADC$ is an augmented chain complex
	\[
	\dots \xrightarrow{d} K_2 \xrightarrow{d} K_1 \xrightarrow{d} K_0 \xrightarrow{e} \ZZ
	\]
	together with distinguished submonoids $K_i^* \subset K_i$ (which do not necessarily satisfy $d(K_i^*) \subset K^*_{i-1}$), and a moprhism $f : K \to L$ is an augmentation-preserving chain map $f_i : K_i \to L_i$ such that $f_i(K_i^*) \subset L_i^*$.
	
	In \cite{Steiner:omega}, Steiner constructed a pair of adjoint functors $\lambda \dashv \nu$ between $\ADC$ and $\omegaCat$, whose right adjoint part $\nu : \ADC \to \omegaCat$ sends an augmented directed complex $K$ to the $\omega$-category $\nu K$ whose cells are families
	\[
	a = \begin{pmatrix}
		a_0^- & a_1^- & \dots & \\
		a_0^+ & a_1^+ & \dots & 
	\end{pmatrix}
	\]
	with $a_i^-,a_i^+ \in K_i^*$ (which are to be thought of as representing the atomic decompositions of the $i$-source/target of $a$) satisfying certain conditions.
	We refer the reader to \cite[Definitions 2.6 and 2.8]{Steiner:omega} for the precise construction, and just assert that such $a$ is a $k$-cell if and only if $a_i^- = a_i^+ = 0$ for all $i >k$.
	By \cite[Theorem 5.6]{Steiner:omega}, this functor $\nu$ is fully faithful when restricted to those augmented directed complexes which admit a \emph{strongly loop-free}, \emph{unital basis} (see \cite[Definitions 3.1, 3.4, and 3.6]{Steiner:omega}).
	
	For $m \ge 0$, let $K[m]$ denote the following augmented directed complex:
	\begin{itemize}
		\item $K[m]_i$ is the free abelian group on the set $\Deltainj\bigl([i],[m]\bigr)$ of injective maps $\beta : [i] \to [m]$ in $\Delta$;
		\item $K[m]^*_i$ is the submonoid (freely) generated by $\Deltainj\bigl([i],[m]\bigr)$;
		\item the differential $d$ is given by $d(\beta) = \sum_{0 \le j \le i}(-1)^j(\beta\delta_j)$; and
		\item the augmentation $e$ sends each $\beta \in \Deltainj\bigl([0],[m]\bigr)$ to $1$.
	\end{itemize}
	By \cite[Theorem 3.2]{Steiner:orientals}, the $\omega$-category $\nu K[m]$ is isomorphic to $\O_m$ .
	In particular, the unique atomic $m$-cell in $\O_m$ corresponds to the \emph{atom} $\langle \id_{[m]}\rangle$ associated to $\id_{[m]} \in \Deltainj\bigl([m],[m]\bigr)$ in the sense of \cite[Definition 3.2]{Steiner:omega}.
	This atom has the form
	\[
	\langle \id_{[m]}\rangle = \begin{pmatrix}
		a_0^- & a_1^- & \dots & a_{m-1}^- & \id_{[m]} & 0 & \dots\\
		a_0^+ & a_1^+ & \dots & a_{m-1}^+ & \id_{[m]} & 0 & \dots
	\end{pmatrix}.
	\]

	Now fix $m,n \ge 0$.
	Then since both $K[m]$ and $K[n]$ admit a strongly loop-free, unital basis by \cite[Theorem 3.1]{Steiner:orientals}, we have a chain of bijections
	\[
	\omegaCat(\O_m,\O_n) \cong \omegaCat\bigl(\nu K[m], \nu K[n]\bigr) \overset{\nu^{-1}}\cong \ADC\bigl(K[m],K[n]\bigr).
	\]
	It is now easy to check that an $\omega$-functor $\O_m \to \O_n$ sends the atomic $m$-cell in $\O_m$ to an $(m-1)$-cell in $\O_n$ if and only if the corresponding map $K[m] \to K[n]$ in $\ADC$ sends the basis element $\id_{[m]} \in K[m]_m$ to $0$.
	
	Finally, we recall how Steiner relates $\ADC\bigl(K[m],K[n]\bigr)$ to $\O(m,n)$.
	Given any $\alpha : [m] \to [n]$ in $\Delta$, consider the function $\Deltainj\bigl([i],[m]\bigr) \to K[n]_i$ given by
	\[
	\beta \mapsto \begin{cases}
		\alpha\beta & \text{if $\alpha\beta$ is injective,}\\
		0 & \text{otherwise.}
	\end{cases}
	\]
	Linearly extending this function yields a homomorphism $\alpha_i : K[m]_i \to K[n]_i$, and $(\alpha_i)_{i \ge 0}$ forms a chain map between the underlying chain complexes of $K[m]$ and $K[n]$.
	Moreover, by \cite[Theorem 4.1]{Steiner:orientals}, linearly extending the assignation $\alpha \mapsto (\alpha_i)_{i \ge 0}$ yields an isomorphism from $\ZZ\Delta\bigl([m],[n]\bigr)$ to the abelian group of chain maps $K[m] \to K[n]$.
	Now it is straightforward to check that, given $x \in \ZZ\Delta\bigl([m],[n]\bigr)$, the corresponding chain map $f : K[m] \to K[n]$ preserves the augmentation if and only if $x$ satisfies (O1), and also $f$ preserves the distinguished submonoids if and only if $x$ satisfies (O2).
	Thus we have the desired bijection
	\[
	\ADC\bigl(K[m],K[n]\bigr) \cong \O(m,n).
	\]
	The lemma now follows by observing that $f$ sends the basis element $\id_{[m]} \in K[m]_m$ to $0$ if and only if the corresponding $x$ is a linear combination of degenerate maps $[m] \to [n]$.
\end{proof}

With this characterisation of marked simplices, we can show the two notions of composition to be equivalent.

\begin{proposition}\label{filler complicial}
	Let $x,y \in \O(m,n)$ and suppose $x\face{k-1} = y\face{k}$.
	Then there exists a unique map $\asimp{m+1}{k} \to \O(-,n)$ that sends $\face{k-1}$ and $\face{k+1}$ to $y$ and $x$ respectively, namely the one that picks out $x \filler k y$.
\end{proposition}

\begin{proof}
	Let $X \subset \asimp {m+1} {k}$ be the minimal regular subset containing $\face{k-1}$ and $\face{k+1}$.
	(Note that $X$ has no non-degenerate marked simplices.)
	We first show that the inclusion $X \incl \asimp {m+1} k$ is an inner complicial anodyne extension.
	
	To see this, note that the simplices $\alpha$ in $\asimp{m+1}k\setminus X$ are precisely those $\alpha : [p] \to [m+1]$ whose images contain both $k-1$ and $k+1$.
	Thus we may partition the set of non-degenerate simplices in $\asimp{m+1}k\setminus X$ into pairs of the form $\{\alpha,\alpha\face i\}$ where $\alpha : [p] \to [m+1]$ satisfies $\alpha(i)=k$ (and necessarily $\alpha(i-1) = k-1$ and $\alpha(i+1) = k+1$).
	Moreover, it is straightforward to see that $\alpha$ defines a regular map of marked simplicial sets $\alpha : \asimp p i \to \asimp{m+1}k$.
	It follows that the inclusion $X \incl \asimp{m+1}k$ may be obtained by filling $\horn p i$ for each such $\alpha$ in increasing order of $p$.
	This exhibits the inclusion as an inner complicial anodyne extension.
	
	By assumption, the simplices $x,y$ specify a map $X \to \O(-,n)$.
	Since $\O(-,n)$ is a complicial set, this map extends uniquely to $\asimp{m+1}k$ by the above argument.
	It now remains to check that the map picking out $x \filler k y$ is indeed such an extension.
	
	Since we have $x \filler k y \in \O(m+1,n)$ by \cref{O pasting}, at least $x \filler k y$ specifies a well-defined simplicial map $\Delta[m+1] \to \O(-,n)$, which moreover satisfies $(x \filler k y)\face{k-1}=y$ and $(x \filler k y)\face{k+1}=x$ by \cref{filler basic}.
	To see that this map respects the marking, let $\beta$ be a non-degenerate, marked $p$-simplex in $\asimp{m+1}k$.
	Then there exists $i \in [p]$ such that $\beta(i-1) = k-1$, $\beta(i) = k$, and $\beta(i+1) = k+1$.
	Since each simplicial operator in $\supp(x \filler k y)$ is degenerate at either $k-1$ or $k$ by \cref{filler}, it follows that each simplicial operator in $\supp\bigl((x \filler k y)\beta\bigr)$ is degenerate at either $i-1$ or $i$.
	Thus $(x \filler k y)\beta$ (which is the image of the marked face $\beta$) is indeed marked by \cref{marked}.
	This completes the proof.
\end{proof}

\subsection{Join of marked simplicial sets}
Recall that the category $\Delta_+$ admits the \emph{join} operation $\join$ which acts on the objects by
\[
[m] \join [n] = [m+1+n]
\]
and this extends to the presheaves, \emph{i.e.}\;augmented simplicial sets, by Day convolution.
Now the \emph{join} $X \oplus Y$ of two (non-augmented) simplicial sets $X$ and $Y$ is the underlying simplicial set of the join of the trivial augmentations of $X$ and $Y$.
More explicitly, we have
\[
(X \join Y)_n = \coprod_{\substack{i,j \ge -1,\\i+1+j=n}}X_i \times Y_j\
\]
for each $n \ge 0$, where $X_{-1} = Y_{-1} = \{*\}$ is a singleton.

The marked version of the join construction is the following.

\begin{definition}
	The \emph{join} $(X,tX) \join (Y,tY)$ of two marked simplicial set $(X,tX)$ and $(Y,tY)$ is the join $X \join Y$ wherein
	\[
	(x,y) \in X_i \times Y_j \subset (X \join Y)_n
	\]
	is marked if and only if either $x \in tX_i$ or $y \in tY_j$ where $tX_{-1}=tY_{-1} = \varnothing$ by convention.
\end{definition}

This join operation interacts well with the inner complicial anodyne extensions, and in particular the following proposition holds by {\cite[Observation 40]{Verity:I}}.

\begin{proposition}\label{join}
	For any inner complicial anodyne extension $f : X \to Y$ and any marked simplicial set $Z$, the join
	\[
	f \join \id : X \join Z \to Y \join Z
	\]
	is an inner complicial anodyne extension.
\end{proposition}

\section{Main theorem}

The purpose of this paper is to prove the following theorem.

\begin{theorem}\label{main}
	The map $\iota_n : \Delta[n] \to \O(-,n)$ picking out the simplex $(0,\dots,n)$ is an inner complicial anodyne extension for any $n \ge 0$.
\end{theorem}

This map $\iota_n$ sends each $m$-simplex $\alpha : [m] \to [n]$ to itself regarded as an element of $\O(m,n)$, that is,
\[
\iota_n(\alpha)_\beta = \begin{cases}
	1 & \text{if }\beta=\alpha,\\
	0 & \text{otherwise.}
\end{cases}
\]

\subsection{Factorisation of $\iota_n$}

We can immediately observe the following.

\begin{lemma}
	The map $\iota_n$ is a regular monomorphism.
\end{lemma}
\begin{proof}
	It is clear from the above explicit description of $\iota_n$ that it is a monomorphism.
	To see that $\iota_n$ is regular, consider an $m$-simplex $\alpha : [m] \to [n]$ in $\Delta[n]$.
	Then by the definition of $\Delta[n]$ as a marked simplicial set, $\alpha$ is marked if and only if it is degenerate, which is equivalent to $\iota_n(\alpha)$ being marked by \cref{marked}.
\end{proof}

The proof of \cref{main} proceeds by induction on $n$.
Note that the base case is trivial since both $\Delta[0]$ and $\O(-,0)$ are the terminal marked simplicial set.
So fix $n \ge 1$.
We factorise the map $\iota_n$ through the following intermediate object.
\begin{definition}
	Let $A[n] \subset \O(-,n)$ be the regular subset consisting of those $x \in \O(m,n)$ such that $\alpha^{-1}(n)$ has the same cardinality for all $\alpha \in \supp(x)$.
\end{definition}

The motivation behind this definition is the following.

\begin{lemma}\label{iota join}
	The map $\iota_n : \Delta[n] \to \O(-,n)$ factors through the subobject $A[n]$.
	Moreover, the factor $\Delta[n] \to A[n]$ is isomorphic to
	\[
	\begin{tikzcd}[column sep = huge]
	\Delta[n] \cong \Delta[n-1] \join \Delta[0]
	\arrow [r, "\iota_{n-1}\join\id"] &
	\O(-,n-1) \join \Delta[0].
	\end{tikzcd}
	\]
\end{lemma}

\begin{proof}
	Since $\supp(\iota_n(\alpha)) = \{\alpha\}$ is a singleton for any $\alpha : [m] \to [n]$, clearly $\iota_n$ factors through $A[n]$.
	It is straightforward to check that the map
	\[
	\O(-,n-1) \join \Delta[0] \to A[n]
	\]
	that sends each $m$-simplex
	\[
	(x,\gamma)\in \O(i,n-1) \times \Delta_+\bigl([j],[0]\bigr)
	\]
	(where $i,j \ge -1$ satisfy $i+1+j=m$) to $\sum_\beta x_\beta \cdot (\beta \join \gamma)$ is an isomorphism of simplicial sets under $\Delta[n]$.
	
	By \cref{marked}, an $m$-simplex $(x,\gamma)$ is marked in $\O(-,n-1) \join \Delta[0]$ if and only if either all $\beta \in \supp(x)$ are degenerate or $j \ge 1$.
	This is equivalent to $\beta \join \gamma$ being degenerate for all $\beta \in \supp(x)$, which is in turn equivalent to $\sum_\beta x_\beta \cdot (\beta \join \gamma)$ being marked in $A[n]$ again by \cref{marked}.
	This completes the proof.
\end{proof}

\begin{remark}
	Although the definition of $A[n]$ is combinatorially simple, it seems much harder to understand it geometrically.
	For example, consider $A[3]$.
	Since $\face 1 = (0,2,3)$ takes the value $3$ whereas $\face 3 = (0,1,2)$ does not, it is straightforward to see that there can be no $2$-simplex in $A[3]$ involving both of these odd faces.
	On the other hand, one would expect to find a pasting of the even faces $\face 0 = (1,2,3)$ and $\face 2 = (0,1,3)$ in $A[3]$, and indeed we have
	\[
	(0,1,3)-(1,1,3)+(1,2,3) \in A[3]_2.
	\]
	As an $\omega$-functor $\O_2 \to \O_3$, this element may be visualised as left below.
	\[
	\begin{tikzpicture}
		\filldraw
		(0,0) circle [radius = 1pt]
		(1.5,1) circle [radius = 1pt]
		(2,0) circle [radius = 1pt];
		
		\draw[->] (0.15,0) -- (1.85,0);
		\draw[->] (0.05,0.1) -- (0.5,1) -- (1.35,1);
		\draw[->] (1.55,0.9) -- (1.95,0.1);
		
		\draw[->, double] (1,0.3) -- (1,0.7);
	\end{tikzpicture}
	\hspace{30pt}
	\begin{tikzpicture}
		\filldraw
		(0,0) circle [radius = 1pt]
		(0.5,1) circle [radius = 1pt]
		(2,0) circle [radius = 1pt];
		
		\draw[->] (0.15,0) -- (1.85,0);
		\draw[->] (0.05,0.1) -- (0.45,0.9);
		\draw[->] (0.6,1) -- (1.5,1) -- (1.95,0.1);
		
		\draw[->, double] (1,0.3) -- (1,0.7);
	\end{tikzpicture}
	\]
	However, if we reparametrise its boundary so that it now looks like right above, the corresponding element
	\[
	(0,1,3)-(1,1,3)+(1,2,3)-(1,2,2)+(1,1,2) \in \O(2,3)
	\]
	is no longer in $A[3]_2$.
	As these examples exhibit, whether a given $\omega$-functor $\O_m \to \O_n$ belongs to $A[n]_m$ depends on its action on all dimensions, and we could not come up with any geometric intuition.
\end{remark}

\subsection{Rank and level}

Our treatment of the factor $A[n] \incl \O(-,n)$ makes use of the following notions of \emph{rank} and \emph{level}.

	\begin{definition}\label{rank}
		Let $x$ be an $m$-simplex in $\O(-,n) \setminus A[n]$.
		We define the \emph{rank} of $x$ to be
		\[
		\rank(x) \defeq \min\left(\bigcup_{\alpha \in \supp(x)}\alpha^{-1}(n)\right).
		\]
		(This is well defined because $\alpha^{-1}(n) = \varnothing$ for all $\alpha \in \supp(x)$ would imply $x \in A[n]_m$.)
		If $x$ has rank $r$, we set
		\[
			\check x_\alpha \defeq \begin{cases}
				x_\alpha, & \text{if }\alpha(r)<n,\\
				0, & \text{if }\alpha(r)=n,
			\end{cases}\quadand
			\bbar x_\alpha \defeq \begin{cases}
				0, & \text{if }\alpha(r)<n,\\
				x_\alpha, & \text{if }\alpha(r)=n.
			\end{cases}
		\]
		for $\alpha : [m] \to [n]$ so that $\check x, \bbar x \in \ZZ\Delta\bigl([m],[n]\bigr)$ satisfy $x = \check x + \bbar x$.
	\end{definition}

The decorations \,$\check{}$\, and \,$\bbar{}$\, correspond to the inequality $<$ and the equality $=$ respectively, which should help the reader remember which is which.

\begin{example}\label{example}
	Consider the $2$-simplices
	\[
	\begin{split}
	x &= (0,1,2) - (1,1,2) + (1,2,2),\\
	w &= (0,1,1) - (1,1,1) + (1,1,2)
	\end{split}
	\]
	in $\O(-,2) \setminus A[2]$; $x$ corresponds to \cref{Omn} middle, and $w = (0,1) \filler 1 (1,2)$ witnesses the composition of the unique non-trivial composable pair in $\O_2$.
	We have $\rank(x) = 1$ and $\rank(w) = 2$, so
	\[
	\begin{alignedat}{3}
		\check x &= (0,1,2) - (1,1,2),& & \quad &
		\bbar x &= (1,2,2),\\
		\check w &= (0,1,1) - (1,1,1),& & \quad &
		\bbar w &= (1,1,2).
	\end{alignedat}
	\]
\end{example}

\begin{remark}
	In \cref{rl} (and also in \cref{dag,uv}), a row labelled ``$y$'' depicts what typical $\alpha \in \supp(y)$ looks like.
	Any shaded area indicates a sub-interval of $[m]$ on which all $\alpha \in \supp(y)$ are constant.
	For example, the lower half of \cref{rl} depicts the fact that we have $\alpha(i) = n$ for all $\alpha \in \supp(\bbar x)$ and for all $r \le i \le m$.
	(Note however that we might have $r=m$, in which case the shaded area really consists of a single $n$.)
	Question marks indicate where we know nothing (or very little) about the values of $\alpha$.
\end{remark}

\begin{remark}
	In \cite[Definition 4.2]{Steiner:universal}, Steiner defined the rank of $x \in \O(m,n)$ to be the smallest $r$ such that $x(r) = x(m)$.
	So for any $x$ in $\O(-,n) \setminus A[n]$ (which must satisfy $x(m) = n$), the rank of $x$ in our sense is at most that of Steiner's.
	We suspect that the two might in fact coincide, but we did not attempt to prove it.
\end{remark}

\begin{definition}
	Let $x$ be an $m$-simplex in $\O(-,n) \setminus A[n]$ with rank $r$.
	We define the \emph{level} of $x$ to be
	\[
	\level(x) \defeq \min\bigl\{i \in [m] :  \alpha(i) = \alpha(r)\text{ for all }\alpha \in \supp(\check x)\bigr\}.
	\]
\end{definition}

\begin{example}
	For $x$ and $w$ from \cref{example}, we have $\level(x) = \level (w) = 1$.
\end{example}

\begin{remark}
	If $x$ is an $m$-simplex in $\O(-,n) \setminus A[n]$ with rank $r$ and level $\ell$, then we must have $\ell \le r$.
	This is why we have $\ell$ to the $\ell$eft and $r$ to the $r$ight in \cref{rl}, though the figure is slightly misleading since we might have $\ell = r$.
	The symbol $\star$ (as well as $\dagger$, $\lozenge$, and $\blacklozenge$ which will appear in \cref{dag,uv}) is a placeholder whose actual value depends on the specific simplicial operator.
\end{remark}

\begin{figure}
	\begin{tikzpicture}[scale = 0.8]
		
		\filldraw[gray!20!white]
		(2.7,0.7) rectangle (5.3,1.3)
		(4.7,-0.3) rectangle (6.8,0.3);
		
		\draw[gray!50!white]
		(1.2,0.7) rectangle (6.8,1.3)
		(1.2,-0.3) rectangle (6.8,0.3);
		
		\node at (0.5,1) {$\check x$};
		\node at (0.5,0) {$\bbar x$};
		
		\node (ulb) at (1.5,1) {$?$};
		\node (ule) at (2.5,1) {$?$};
		\node (umb) at (3,1) {$\star$};
		\node (ume) at (5,1) {$\star$};
		\node (urb) at (5.5,1) {$?$};
		\node (ure) at (6.5,1) {$?$};
		
		\node (llb) at (1.5,0) {$?$};
		\node (lme) at (4.5,0) {$?$};
		\node (lrb) at (5,0) {$n$};
		\node (lre) at (6.5,0) {$n$};
		
		\draw[thick, dotted]
		(ulb.east) -- (ule.west)
		(umb.east) -- (ume.west)
		(urb.east) -- (ure.west)
		(llb.east) -- (lme.west)
		(lrb.east) -- (lre.west);
		
		\node[scale = 0.9] at (3,1.7) {$\ell$};
		\node[scale = 0.9] at (5,1.7) {$r$};
		\node[scale = 0.9] at (1.5,1.7) {$0$};
		\node[scale = 0.9] at (6.5,1.7) {$m$};
	\end{tikzpicture}
	\caption{$r$, $\check x$, $\bbar x$ and $\ell$}\label{rl}
\end{figure}

The following observation will be useful.

\begin{lemma}\label{level}
	Let $x$ be an $m$-simplex in $\O(-,n) \setminus A[n]$ with rank $r$ and level $\ell$.
	Then we have $\ell \ge 1$.
\end{lemma}

\begin{proof}
	Suppose for contradiction that $\ell=0$.
	The following series of claims leads to the conclusion
	\[
	\sum_{\alpha : [m] \to [n]} x_\alpha = 0
	\]
	which contradicts the assumption that $x$ satisfies (O1).
	
	Consider $y = x  (0,\dots,r) \in \O(r,n)$ and $z = x  (r,\dots,m) \in \O(m-r,n)$.
	Note that we have
	\[
	y_{\beta} = \sum_{\substack{\alpha : [m] \to [n]\\\alpha(0,\dots,r) = \beta}}x_{\alpha} \quadand z_\gamma = \sum_{\substack{\alpha : [m] \to [n]\\\alpha(r,\dots,m) = \gamma}}x_{\alpha}
	\]
	for any $\beta : [r] \to [n]$ and $\gamma : [m-r] \to [n]$.
	
	\begin{claim}\label{claim1a}\leavevmode
		\begin{itemize}
			\item [(y)] For any $\beta : [r] \to [n]$ with $\beta(r) = n$, we have $y_\beta = x_{\beta\degen r^{m-r}}$.
			\item [(z)] For any $\gamma : [m-r] \to [n]$ with $\gamma(0) < n$, we have $z_\gamma = x_{\gamma\degen 0^r}$.
		\end{itemize}
		
	\end{claim}
	\begin{proof}[Proof of Claim]
		(y) It suffices to observe that, if $\beta(r) = n$, the only solution to the equation $\alpha(0,\dots,r) = \beta$ in $\Delta\bigl([m],[n]\bigr)$ is $\alpha = \beta\degen r ^{m-r}$.
		
		(z) Fix $\gamma : [m-r] \to [n]$ with $\gamma(0) < n$ and consider the equation $\alpha(r,\dots,m) = \gamma$.
		Note that if $\alpha \in \supp(x)$ is a solution, then $\alpha(r) = \gamma(0) < n$ and so $\alpha \in \supp(\check x)$.
		Since we are assuming $\ell = 0$, we must have $\alpha(i) = \alpha(r)$ for all $0 \le i \le r$, hence $\alpha = \gamma\degen 0^r$.
	\end{proof}
	
	\begin{claim}\label{claim1b}\leavevmode
		\begin{itemize}
			\item[(y)] We have $\bigl\{\alpha(0,\dots,r) : \alpha \in \supp(\bbar x)\bigr\} \subset \supp(y)$.
			\item[(z)] We have $\bigl\{\alpha(r,\dots,m) : \alpha \in \supp(\check x)\bigr\} \subset \supp(z)$.
		\end{itemize}
		
	\end{claim}
	\begin{proof}[Proof of Claim]
		For any $\alpha \in \supp(\bbar x)$, we have $\alpha(0,\dots,r)(r) = \alpha(r) = n$ and $\alpha(0,\dots,r)\degen r^{m-r} = \alpha$.
		So substituting $\beta = \alpha(0,\dots,r)$ into \cref{claim1a}(y) yields
		\[
		y_{\alpha(0,\dots,r)} = x_{\alpha(0,\dots,r)\degen r^{m-r}} = x_\alpha \neq 0.
		\]
		This proves (y), and the proof for (z) is similar.
	\end{proof}
	\begin{claim}\label{claim1c}\leavevmode
		\begin{itemize}
			\item [(y)] We have $\displaystyle \sum_{\substack{\alpha : [m] \to [n]\\\alpha(r) < n}}x_\alpha = \sum_{\substack{\beta : [r] \to [n]\\\beta(r) < n}} y_\beta.$
			\item [(z)] We have $\displaystyle \sum_{\substack{\alpha : [m] \to [n]\\\alpha(r) = n}}x_\alpha = z_{(n,\dots,n)}.$
		\end{itemize}
		
	\end{claim}
	\begin{proof}[Proof of Claim]
		(y) Since we have
		\[
		\sum_{\alpha : [m] \to [n]}x_\alpha = \sum_{\beta : [r] \to [n]} \left(\sum_{\substack{\alpha : [m] \to [n]\\\alpha(0,\dots,r) = \beta}}x_{\alpha}\right) = \sum_{\beta : [r] \to [n]}y_\beta,
		\]
		it suffices to prove
		\[
		\sum_{\substack{\alpha : [m] \to [n]\\\alpha(r) = n}}x_\alpha = \sum_{\substack{\beta : [r] \to [n]\\\beta(r) = n}} y_\beta.
		\]
		The latter equality follows from \cref{claim1a}(y) because a simplicial operator $\alpha : [m] \to [n]$ satisfies $\alpha(r) = n$ if and only if it is of the form $\alpha = \beta\degen r^{m-r}$ for some $\beta : [r] \to [n]$ with $\beta(r) = n$.
		
		(z) This equality follows from the observation that a simplicial operator $\alpha : [m] \to [n]$ satisfies $\alpha(r) = n$ if and only if $\alpha(r,\dots,m) = (n,\dots,n)$.
	\end{proof}
	
	\begin{claim}\label{claim1d}\leavevmode
		\begin{itemize}
			\item [(y)] We have $y_\beta = 0$ for any $\beta : [r] \to [n]$ with $\beta(r) < n$.
			\item [(z)] We have $z_{(n,\dots,n)} = 0$.
		\end{itemize}
		
	\end{claim}
	\begin{proof}[Proof of Claim]
		(y) 
		It follows from our assumption $\ell = 0$ that $y_\beta = 0$ for any non-constant $\beta : [r] \to [n]$ with $\beta(r) < n$.
		So it suffices to show $y_{(k,\dots,k)} = 0$ for $0 \le k \le n-1$.
		Observe that, for fixed $0 \le k \le n-1$, the constant operator $(k,\dots,k)$ is the only potential element of $\supp(y)$ that sends $r$ to $k$.
		So the coefficient of $(k)$ in $y(r)$ is $y_{(k,\dots,k)}$.
		But $y(r) = (n)$ by \cref{terminus} and \cref{claim1b}(y), so we must have $y_{(k,\dots,k)} = 0$.
		
		(z) Since $(n,\dots,n)$ is the only simplicial operator $[m-r] \to [n]$ that sends $0$ to $n$, the coefficient of $(n)$ in $z(0)$ is $z_{(n,\dots,n)}$.
		But $z(0) < n$ by \cref{terminus} and \cref{claim1b}(z), so we must have $z_{(n,\dots,n)} = 0$.
	\end{proof}
	As we asserted at the beginning of this proof, we can now deduce
	\[
	\begin{aligned}
		\sum_{\alpha : [m] \to [n]}x_\alpha &= \sum_{\substack{\alpha : [m] \to [n]\\\alpha(r) < n}}x_\alpha + \sum_{\substack{\alpha : [m] \to [n]\\\alpha(r) = n}}x_\alpha & & \\
		&= \sum_{\substack{\beta : [r] \to [n]\\\beta(r) < n}} y_\beta + z_{(n,\dots,n)} & & \text{(\cref{claim1c})} \\
		&= 0 & & \text{(\cref{claim1d})}
	\end{aligned}
	\]
	which leads to the desired contradiction.
	This completes the proof.
\end{proof}

\subsection{Parent-child pairing}

We will eventually show that the inclusion $A[n] \incl \O(-,n)$ can be obtained by filling horns and extending marking.
The following constructions provide the interior-face pairing for those horns.

\begin{definition}
	Let $x$ be a simplex in $\O(-,n)\setminus A[n]$ with rank $r$ and level $\ell$.
	Consider the condition
	\begin{itemize}
		\item[$(\ddagger)$] $\bbar x$ is degenerate at $\ell-1$.
	\end{itemize}
	If $x$ satisfies $(\ddagger)$, we define its \emph{child} to be $x\face \ell$.
	If $x$ does not satisfy $(\ddagger)$, we define its \emph{parent} to be $\check x \degen{\ell} + \bbar x \degen{\ell-1}$.
\end{definition}

\begin{example}
	Recall our running examples
	\[
	\begin{split}
		x &= (0,1,2) - (1,1,2) + (1,2,2),\\
		w &= (0,1,1) - (1,1,1) + (1,1,2)
	\end{split}
	\]
	in $\O(-,2) \setminus A[2]$.
	Since $\bbar x = (1,2,2)$ is not degenerate at $\level(x)-1 = 0$, $x$ does not satisfy $(\ddagger)$.
	Its parent is
	\[
	\check x \degen 1 + \bbar x \degen 0 = (0,1,1,2) - (1,1,1,2) + (1,1,2,2).
	\]
	On the other hand, $\bbar w = (1,1,2)$ is degenerate at $\level(w)-1 = 0$, so $w$ satisfies $(\ddagger)$.
	Its child is
	\[
	w\face 1 = (0,1) - (1,1) + (1,2).
	\]
\end{example}

\begin{remark}\label{l<r}
	Since any $x$ in $\O(-,n) \setminus A[n]$ has positive level by \cref{level}, it always makes sense to ask whether $x$ satisfies $(\ddagger)$.
	Note that in the case $x$ does satisfy $(\ddagger)$, we must have $\ell<r$ since $\alpha(r-1) < n = \alpha(r)$ for all $\alpha \in \supp(\bbar x)$.
	(It is however possible to have $\ell = r-1$, in which case the gap between the rightmost $\dagger$ and the leftmost $n$ in \cref{dag} is empty.)
\end{remark}

\begin{figure}
	\begin{tikzpicture}[scale = 0.8]
		
		\filldraw[gray!20!white]
		(2.2,-0.3) rectangle (3.3,0.3)
		(2.7,0.7) rectangle (5.3,1.3)
		(4.7,-0.3) rectangle (6.8,0.3);
		
		\draw[gray!50!white]
		(0.7,0.7) rectangle (6.8,1.3)
		(0.7,-0.3) rectangle (6.8,0.3);
		
		\node at (0,1) {$\check x$};
		\node at (0,0) {$\bbar x$};
		
		\node (ulb) at (1,1) {$?$};
		\node (ule) at (2.5,1) {$?$};
		\node (umb) at (3,1) {$\star$};
		\node (ume) at (5,1) {$\star$};
		\node (urb) at (5.5,1) {$?$};
		\node (ure) at (6.5,1) {$?$};
		
		\node (llb) at (1,0) {$?$};
		\node (lle) at (2,0) {$?$};
		\node at (2.5,0) {$\dagger$};
		\node at (3,0) {$\dagger$};
		\node (lmb) at (3.5,0) {$?$};
		\node (lme) at (4.5,0) {$?$};
		\node (lrb) at (5,0) {$n$};
		\node (lre) at (6.5,0) {$n$};
		
		\draw[thick, dotted]
		(ulb.east) -- (ule.west)
		(umb.east) -- (ume.west)
		(urb.east) -- (ure.west)
		(llb.east) -- (lle.west)
		(lmb.east) -- (lme.west)
		(lrb.east) -- (lre.west);
		
		\node[scale = 0.9] at (3,1.7) {$\ell$};
		\node[scale = 0.9] at (5,1.7) {$r$};
		\node[scale = 0.9] at (1,1.7) {$0$};
		\node[scale = 0.9] at (6.5,1.7) {$m$};
	\end{tikzpicture}
	\caption{The condition $(\ddagger)$}\label{dag}
\end{figure}

From now on, we will focus on the \emph{non-degenerate} simplices in $\O(-,n) \setminus A[n]$.

\begin{lemma}\label{dagchild}
	Let $x$ be a non-degenerate $m$-simplex in $\O(-,n) \setminus A[n]$ with rank $r$ and level $\ell$, and suppose that $x$ satisfies $(\ddagger)$.
	Then its child $y = x\face \ell$ is a non-degenerate $(m-1)$-simplex in $\O(-,n) \setminus A[n]$ such that:
	\begin{itemize}
		\item $\rank(y)=r-1$;
		\item $\level(y)=\ell$; and
		\item $y$ does not satisfy $(\ddagger)$.
	\end{itemize}
Moreover we have $\check y = \check x \face \ell$ and $\bbar y = \bbar x \face \ell$.
\end{lemma}
\begin{proof}
	We first prove that any two simplicial operators $\alpha,\beta \in \supp(x)$ with $\alpha\face\ell = \beta\face\ell$ must be equal.
	Since we have $\ell < r$ as observed in \cref{l<r}, the assumption $\alpha\face\ell = \beta\face\ell$ in particular implies $\alpha(r)=\beta(r)$, so either $\alpha,\beta\in \supp(\check x)$ or $\alpha,\beta \in \supp(\bbar x)$.
	In the former case, it follows from $\ell < r$ that both $\alpha$ and $\beta$ are degenerate at $\ell$, so we have
		\[
		\alpha = \alpha\face \ell \degen \ell = \beta\face \ell \degen \ell = \beta.
		\]
	In the latter case, it follows from $(\ddagger)$ that both $\alpha$ and $\beta$ are degenerate at $\ell-1$, so
	\[
	\alpha = \alpha\face \ell \degen {\ell-1} = \beta\face \ell \degen {\ell-1} = \beta.
	\]
In either case, we have $\alpha = \beta$ as desired.
It follows that no terms in $x$ cancel out after applying $\face\ell$; more precisely, the support of $y = x \face \ell$ is given by
\[
\supp(y) = \bigl\{\alpha\face \ell : \alpha \in \supp(x)\bigr\}.
\]

Everything we asserted about $y$ follows from this description of $\supp(y)$ and the fact $\ell < r$.

Firstly, observe that we have $|(\alpha\face\ell)^{-1}(n)| = |\alpha^{-1}(n)|$ for all $\alpha \in \supp(x)$.
So $x \notin A[n]_m$ implies $y \notin A[n]_{m-1}$.

The equations $\rank(y)=r-1$, $\check y = \check x \face \ell$, $\bbar y = \bbar x \face \ell$, and $\level(y)=\ell$ are all straightforward to verify (in this order).

Recall that a simplex in $\O(-,n)$ is degenerate at some $k$ if and only if all simplicial operators in its support are degenerate at $k$.
So, if $y$ is degenerate at some $k < r$ (respectively $k \ge r$) then $x$ must be degenerate at $k$ (respectively $k+1$).
Taking the contrapositive, we deduce that $y$ is non-degenerate.

Finally, observe that $y$ satisfies $(\ddagger)$ if and only if $\bbar x \face \ell$ is degenerate at $\ell$ because we have $\bbar y = \bbar x \face \ell$ and $\level(y) = \ell$.
However, this would imply that $\bbar x$ and hence $x$ are degenerate at $\ell$, which contradicts our assumption that $x$ is non-degenerate.
\end{proof}

\begin{lemma}\label{dagparent}
	Let $x$ be a non-degenerate $m$-simplex in $\O(-,n) \setminus A[n]$ with rank $r$ and level $\ell$, and suppose that $x$ does not satisfy $(\ddagger)$.
	Then its parent $w = \check x \degen{\ell} + \bbar x \degen{\ell-1}$ is a non-degenerate $(m+1)$-simplex in $\O(-,n) \setminus A[n]$ such that:
	\begin{itemize}
		\item $\rank(w) = r+1$;
		\item $\level(w) = \ell$; and
		\item $w$ satisfies $(\ddagger)$.
	\end{itemize}
	Moreover we have $w \face \ell = x$.
\end{lemma}

\begin{figure}
\[
\begin{tikzpicture}[scale = 0.8]
	\filldraw[gray!20!white]
	(2.7,5.7) rectangle (5.3,6.3)
	(4.7,4.7) rectangle (6.8,5.3);
	\draw[gray!50!white]
	(0.7,5.7) rectangle (6.8,6.3)
	(0.7,4.7) rectangle (6.8,5.3);
	\node at (-0.5,6) {$\check x$};
	\node at (-0.5,5) {$\bbar x$};
	\node (ulbt) at (1,6) {$?$};
	\node (ulet) at (2,6) {$?$};
	\node at (2.5,6) {$\lozenge$};
	\node (umbt) at (3,6) {$\star$};
	\node (umet) at (5,6) {$\star$};
	\node (urbt) at (5.5,6) {$?$};
	\node (uret) at (6.5,6) {$?$};
	\node (llbt) at (1,5) {$?$};
	\node (llet) at (2,5) {$?$};
	\node at (2.5,5) {$\dagger$};
	\node at (3,5) {$\blacklozenge$};
	\node (lmbt) at (3.5,5) {$?$};
	\node (lmet) at (4.5,5) {$?$};
	\node (lrbt) at (5,5) {$n$};
	\node (lret) at (6.5,5) {$n$};
	\draw[thick, dotted]
	(ulbt.east) -- (ulet.west)
	(umbt.east) -- (umet.west)
	(urbt.east) -- (uret.west)
	(llbt.east) -- (llet.west)
	(lmbt.east) -- (lmet.west)
	(lrbt.east) -- (lret.west);

	\filldraw[gray!20!white]
	(2.2,2.2) rectangle (3.3,2.8)
	(2.7,3.2) rectangle (5.3,3.8)
	(4.7,2.2) rectangle (6.8,2.8);
	\draw[gray!50!white]
	(0.7,3.2) rectangle (6.8,3.8)
	(0.7,2.2) rectangle (6.8,2.8);
	\node at (-0.5,3.5) {$\check x$};
	\node at (-0.5,2.5) {$\bbar x\face{\ell}\degen{\ell-1}$};
	\node (ulbs) at (1,3.5) {$?$};
	\node (ules) at (2,3.5) {$?$};
	\node at (2.5,3.5) {$\lozenge$};
	\node (umbs) at (3,3.5) {$\star$};
	\node (umes) at (5,3.5) {$\star$};
	\node (urbs) at (5.5,3.5) {$?$};
	\node (ures) at (6.5,3.5) {$?$};
	\node (llbs) at (1,2.5) {$?$};
	\node (lles) at (2,2.5) {$?$};
	\node at (2.5,2.5) {$\dagger$};
	\node at (3,2.5) {$\dagger$};
	\node (lmbs) at (3.5,2.5) {$?$};
	\node (lmes) at (4.5,2.5) {$?$};
	\node (lrbs) at (5,2.5) {$n$};
	\node (lres) at (6.5,2.5) {$n$};
	\draw[thick, dotted]
	(ulbs.east) -- (ules.west)
	(umbs.east) -- (umes.west)
	(urbs.east) -- (ures.west)
	(llbs.east) -- (lles.west)
	(lmbs.east) -- (lmes.west)
	(lrbs.east) -- (lres.west);

	\filldraw[gray!20!white]
	(2.2,0.7) rectangle (5.3,1.3)
	(4.7,-0.3) rectangle (6.8,0.3);
	\draw[gray!50!white]
	(0.7,0.7) rectangle (6.8,1.3)
	(0.7,-0.3) rectangle (6.8,0.3);
	\node at (-0.5,1) {$\check x\face{\ell-1}\degen{\ell-1}$};
	\node at (-0.5,0) {$\bbar x$};
	\node (ulb) at (1,1) {$?$};
	\node (ule) at (2,1) {$?$};
	\node at (2.5,1) {$\star$};
	\node (umb) at (3,1) {$\star$};
	\node (ume) at (5,1) {$\star$};
	\node (urb) at (5.5,1) {$?$};
	\node (ure) at (6.5,1) {$?$};
	\node (llb) at (1,0) {$?$};
	\node (lle) at (2,0) {$?$};
	\node at (2.5,0) {$\dagger$};
	\node at (3,0) {$\blacklozenge$};
	\node (lmb) at (3.5,0) {$?$};
	\node (lme) at (4.5,0) {$?$};
	\node (lrb) at (5,0) {$n$};
	\node (lre) at (6.5,0) {$n$};
	\draw[thick, dotted]
	(ulb.east) -- (ule.west)
	(umb.east) -- (ume.west)
	(urb.east) -- (ure.west)
	(llb.east) -- (lle.west)
	(lmb.east) -- (lme.west)
	(lrb.east) -- (lre.west);

	\draw [decorate,decoration={brace,amplitude=5pt}]
	(-1.7,4.6) -- (-1.7,6.4) node [midway, xshift = -15pt] {$x$};
	
	\draw [decorate,decoration={brace,amplitude=5pt}]
	(-1.7,2.1) -- (-1.7,3.9) node [midway, xshift = -15pt] {$u$};
	
	\draw [decorate,decoration={brace,amplitude=5pt}]
	(-1.7,-0.4) -- (-1.7,1.4) node [midway, xshift = -15pt] {$v$};
	
	\node[scale = 0.9] at (3,6.7) {$\ell$};
	\node[scale = 0.9] at (5,6.7) {$r$};
	\node[scale = 0.9] at (1,6.7) {$0$};
	\node[scale = 0.9] at (6.5,6.7) {$m$};
\end{tikzpicture}
\]
\caption{$x$, $u$ and $v$ in \cref{dagparent}}\label{uv}
\end{figure}

\begin{proof}
	We will first assume $w \in \O(m+1,n)$ and prove the other assertions in the lemma (because this part of the proof is similar to that of \cref{dagchild} and the reader should find it easier to follow it before going through a completely different kind of combinatorics).
	
	Let $\alpha,\beta \in \supp(x)$ and suppose that one of the following holds:
	\begin{itemize}
		\item $\alpha,\beta \in \supp(\check x)$ and $\alpha\degen\ell = \beta\degen\ell$;
		\item $\alpha \in \supp(\check x)$, $\beta \in \supp(\bbar x)$ and $\alpha\degen\ell = \beta\degen{\ell-1}$; or
		\item $\alpha,\beta\in\supp(\bbar x)$ and $\alpha\degen{\ell-1} = \beta\degen{\ell-1}$.
	\end{itemize}
	Then, since $\face{\ell}$ is a common section of $\degen{\ell}$ and $\degen{\ell-1}$, acting by $\face\ell$ on any of the equalities yields $\alpha = \beta$.
	It follows that the support of $w = \check x \degen{\ell} + \bbar x \degen{\ell-1}$ is given by
	\[
	\supp(w) = \bigl\{\alpha\degen{\ell} : \alpha \in \supp(\check x)\bigr\} \cup \bigl\{\alpha\degen{\ell-1} : \alpha \in \supp(\bbar x)\bigr\}.
	\]
	
	Now everything (except for $w \in \O(m+1,n)$) follows from this description of $\supp(w)$.
	
	Firstly, choose $\alpha \in \supp(\check x)$ and $\beta \in \supp(\bbar x)$, which necessarily satisfy $|\alpha^{-1}(n)| < |\beta^{-1}(n)|$.
	Note that $\ell \le r$ implies $\alpha(\ell) < n$ and $\beta(\ell-1) < n$.
	We can thus deduce
	\[
	|(\alpha\degen\ell)^{-1}(n)| = |\alpha^{-1}(n)| < |\beta^{-1}(n)| = |(\beta\degen{\ell-1})^{-1}(n)|.
	\]
	Since we know $\alpha\degen\ell, \beta\degen{\ell-1} \in \supp(w)$, this implies $w \notin A[n]_{m+1}$.
	
	It is straightforward to check the equations $w \face\ell = x$, $\rank(w) = r+1$, $\check w = \check x \degen\ell$, $\bbar w = \bbar x\degen{\ell-1}$ and $\level(w) = \ell$.
	The last two of these equalities imply that $w$ satisfies $(\ddagger)$.
	
	Suppose for contradiction that $w$ is degenerate at some $k$.
	\begin{itemize}
		\item If $k < \ell - 1$ then $x$ must be degenerate at $k$.
		\item If $k = \ell-1$ then $\alpha\degen\ell$ and hence $\alpha$ must be degenerate at $\ell-1$ for all $\alpha \in \supp(\check x)$, which contradicts the minimality of $\ell$.
		\item If $k = \ell$ then in particular $\alpha\degen{\ell-1}$ is degenerate at $\ell$ for each $\alpha \in \supp(\bbar x)$, or equivalently each $\alpha \in \supp(\bbar x)$ is degenerate at $\ell-1$.
		This contradicts the assumption that $x$ does not satisfy $(\ddagger)$.
		\item If $k > \ell$ then $x$ must be degenerate at $k-1$.
	\end{itemize}
	Since each case indeed leads to a contradiction, $w$ is non-degenerate.
	
	Now we prove $w \in \O(m,n+1)$.
	Recall that we have $\ell \ge 1$ by \cref{level}.
	Let
	\[
	u = \check x + \bbar x \face \ell \degen{\ell-1} \quadand v = \check x \face{\ell-1}\degen{\ell-1}+\bbar x.
	\]
	(See \cref{uv} for what they look like; note however that in the special case $\ell = r$, we have $\blacklozenge = n$.)
	Then we have $w = u \filler \ell v$ by \cref{decomposition}.
	Thus, thanks to \cref{O pasting}, it suffices to prove $u,v \in \O(m,n)$.
	Both $u$ and $v$ clearly satisfy (O1), and it remains to show that they also satisfy (O2).
	So fix injective maps $\beta : [p] \to [m]$ and $\gamma : [p] \to [n]$ in $\Delta$.
	We will divide the proof into the following six claims.
	
	\begin{claim}\label{claim2a}
		If $\ell \notin \im(\beta)$ then $(u\beta)_\gamma\ge 0$.
	\end{claim}
	\begin{proof}[Proof of Claim]
		Suppose $\ell \notin \im(\beta)$.
		Then $\beta$ can be written as $\beta = \face\ell\beta'$ for some $\beta' : [p] \to [m-1]$.
		So we have
		\[
		\begin{split}
		u\beta &= (\check x + \bbar x \face \ell \degen{\ell-1})\face\ell\beta'\\
		&= \check x \face\ell\beta'+ \bbar x \face \ell \degen{\ell-1}\face\ell\beta'\\
		&= \check x \face\ell\beta'+ \bbar x \face \ell \beta'\\
		&= (\check x + \bbar x)\face\ell\beta'\\
		&= x\beta.
		\end{split}
		\]
		Since $x$ satisfies (O2), we can deduce $(u\beta)_\gamma = (x\beta)_\gamma \ge 0$.
	\end{proof}
	
	\begin{claim}\label{claim2b}
		If $\ell-1 \notin \im(\beta)$ then $(v\beta)_\gamma \ge 0$.
	\end{claim}
	\begin{proof}[Proof of Claim]
		Similar to \cref{claim2a}.
	\end{proof}
	
	\begin{claim}\label{claim2c}
		If $\ell \in \im(\beta)$ then $(\check x\beta)_\gamma \ge 0$. 
	\end{claim}
	\begin{proof}[Proof of Claim]
		Let $k \in [p]$ be the unique integer satisfying $\beta(k) = \ell$.
		If $\gamma(k) = n$ then $(\check x \beta)_\gamma = 0$ by the definition of $\check x$.
		So assume $\gamma(k) < n$.
		
		\begin{case}
		Suppose $k \le p-1$ and $\beta(k+1) \le r$.
		In this case, for any $\alpha \in \supp(\check x)$, we have
		\[
		(\alpha\beta)(k+1) = (\alpha\beta)(k)
		\]
		since $\ell = \beta(k) < \beta(k+1) \le r$.
		Thus $\check x\beta$ is degenerate at $k$, which implies $(\check x\beta)_\gamma = 0$ because $\gamma$ is non-degenerate.
		\end{case}
		
		\begin{case}
		Suppose that we have either $k=p$ or $\beta(k+1) > r$.
		In this case, one can easily check that $\beta' : [p] \to [m]$ defined by
		\[
		\beta'(i) = \begin{cases}
			r, & \text{if }i=k,\\
			\beta(i), & \text{otherwise}
		\end{cases}
		\]
		is an injective simplicial operator.
		Since $\alpha(\ell) = \alpha(r)$ for all $\alpha \in \supp(\check x)$, we have $\check x \beta = \check x \beta'$.
		Moreover, we have $(\bbar x\beta')_\gamma = 0$ since $\gamma(k) < n$ whereas
		\[
		(\alpha\beta')(k) = \alpha(r) = n
		\]
		for all $\alpha \in \supp(\bbar x)$.
		Thus we have
		\[
			(\check x \beta)_\gamma =  (\check x \beta')_\gamma
			= (\check x\beta')_\gamma + (\bbar x\beta')_\gamma
			= (x\beta')_\gamma
			\ge 0
		\]
		where the last inequality follows from the assumption that $x$ satisfies (O2).
		\end{case}
		This completes the proof of \cref{claim2c}.
	\end{proof}

	\begin{claim}\label{claim2d}
		If $\ell-1 \in \im(\beta)$ then $(\bbar x\beta)_\gamma \ge 0$.
	\end{claim}
	\begin{proof}[Proof of Claim]
		Suppose $\ell-1 \in \im(\beta)$.
		If $(\bbar x\beta)_\gamma = 0$ then we are done, so assume $(\bbar x\beta)_\gamma \neq 0$.
		
		\begin{case}
			Suppose $\gamma(p) = n$.
			Then at least one $\alpha \in \supp(\bbar x)$ must satisfy
			\[
			(\alpha\beta)(p) = \gamma(p) = n,
			\]
			which implies $\beta(p) \ge r$.
			If $\beta(p-1) \ge r$ also holds, then for each $\alpha \in \supp(\bbar x)$, we have
			\[
			(\alpha\beta)(p-1) = n = (\alpha\beta)(p).
			\]
			But this is impossible since at least one $\alpha \in \supp(\bbar x)$ must satisfy $\alpha\beta = \gamma$ and $\gamma$ is non-degenerate.
			Thus we must have $\beta(p-1) < r$, and consequently $\beta' : [p] \to [m]$ defined by
			\[
			\beta'(i) = \begin{cases}
				r & \text{if }i=p,\\
				\beta(i) & \text{otherwise}
			\end{cases}
			\]
			is an injective map in $\Delta$.
			It is straightforward to check that each $\alpha \in \supp(\bbar x)$ satisfies $\alpha\beta = \alpha\beta'$ and no $\alpha' \in \supp(\check x)$ satisfies $\alpha'\beta' = \gamma$.
			Therefore we have
			\[
			(\bbar x\beta)_\gamma = (\bbar x\beta')_\gamma = (\bbar x\beta')_\gamma + (\check x \beta')_\gamma = (x\beta')_\gamma \ge 0
			\]
			where the last inequality follows from the assumption that $x$ satisfies (O2).
		\end{case}
		\begin{case}
			Suppose $\gamma(p) < n$.
			In this case, we must have $\beta(p) < r$ for otherwise we would have
			\[
			(\alpha\beta)(p) = n \neq \gamma(p)
			\]
			for all $\alpha \in \supp(\bbar x)$.
			Thus $\beta' : [p+1] \to [m]$ and $\gamma' : [p+1] \to [n]$ given by
			\[
			\beta'(i) = \begin{cases}
				r & \text{if }i=p+1,\\
				\beta(i) & \text{otherwise,}
			\end{cases}
			\quadand
			\gamma'(i) = \begin{cases}
				n & \text{if }i=p+1,\\
				\gamma(i) & \text{otherwise}
			\end{cases}
			\]
			are both injective maps in $\Delta$.
			It is straightforward to check that, for each $\alpha \in \supp(\bbar x)$, we have $\alpha\beta = \gamma$ if and only if $\alpha\beta' = \gamma'$.
			Moreover, no $\alpha' \in \supp(\check x)$ satisfies $\alpha'\beta' = \gamma'$.
			Therefore we have
			\[
			(\bbar x \beta)_\gamma = (\bbar x \beta')_{\gamma'} = (\bbar x \beta')_{\gamma'} + (\check x\beta')_{\gamma'} = (x\beta')_{\gamma'} \ge 0
			\]
			where the last inequality follows from the assumption that $x$ satisfies (O2).
		\end{case}
	This completes the proof of \cref{claim2d}.
	\end{proof}

	\begin{claim}\label{claim2e}
		If $\ell \in \im(\beta)$ then $(\bbar x\face\ell\degen{\ell-1}\beta)_\gamma \ge 0$.
	\end{claim}
	\begin{proof}[Proof of Claim]
		Suppose $\ell \in \im(\beta)$.
		If moreover $\ell-1 \in \im(\beta)$ then  $\bbar x\face\ell\degen{\ell-1}\beta$ is degenerate whereas $\gamma$ is not, which implies $(\bbar x\face\ell\degen{\ell-1}\beta)_\gamma = 0$.
		So suppose $\ell-1 \notin \im(\beta)$.
		Then $\beta$ can be written as $\beta = \face{\ell-1}\beta'$ for some injective $\beta'$ in $\Delta$, hence
		\[
		\bbar x \face \ell \degen{\ell-1}\beta = \bbar x \face \ell \degen{\ell-1}\face{\ell-1}\beta' = \bbar x \face \ell\beta'.
		\]
		Since $\ell \in \im(\beta)$ implies $\ell-1 \in \im(\beta')$, we have $\ell-1 \in \im(\face \ell \beta')$.
		Thus the desired inequality $(\bbar x \face \ell \beta')_\gamma \ge 0$ follows from \cref{claim2d}.
	\end{proof}

	\begin{claim}\label{claim2f}
		If $\ell-1 \in \im(\beta)$ then $(\check x\face{\ell-1}\degen{\ell-1}\beta)_\gamma \ge 0$.
	\end{claim}
	\begin{proof}[Proof of Claim]
		Similarly to the proof of \cref{claim2e}, we may write $\beta$ as $\beta = \face \ell \beta'$ and reduce the desired inequality to \cref{claim2c}.
	\end{proof}
	
	We see that $u$ satisfies (O2) by \cref{claim2a,claim2c,claim2e} and $v$ satisfies (O2) by \cref{claim2b,claim2d,claim2f}.
	Therefore $w = u \filler \ell v$ is indeed an $(m+1)$-simplex in $\O(-,n)$, and this completes the proof of \cref{dagparent}.
\end{proof}

\begin{lemma}\label{dagpairing}
	The set of non-degenerate simplices in $\O(-,n)\setminus A[n]$ can be partitioned into parent-child pairs.
\end{lemma}
\begin{proof}
	By \cref{dagchild}, the child construction defines a function from the set of non-degenerate simplices in $\O(-,n)\setminus A[n]$ satisfying $(\ddagger)$ to the set of those not satisfying $(\ddagger)$.
	Similarly, \cref{dagparent} implies that the parent construction defines a function in the other direction.
	We must show that they are inverse to each other.
	
	We first prove that the parent construction followed by the child construction is the identity.
	Let $x$ be a non-degenerate simplex in $\O(-,n) \setminus A[n]$ that does not satisfy $(\ddagger)$, and let $w$ be its parent.
	Then it follows from \cref{dagparent} that the child of $w$ is indeed $w\face{\level(w)} = w \face \ell = x$.
	
	For the other composite, let $x$ be a non-degenerate simplex in $\O(-,n) \setminus A[n]$ that satisfies $(\ddagger)$, and let $y$ be its child.
	Then it follows from \cref{dagchild} that the parent of $y$ is
	\[
	\check y \degen{\level(y)} + \bbar y \degen {\level(y)-1} = \check x \face \ell \degen{\ell} + \bbar x \face \ell \degen {\ell-1}.
	\]
	Since the condition $(\ddagger)$ implies that $\check x$ is degenerate at $\ell$ and $\bbar x$ is degenerate at $\ell-1$, this parent is indeed $x$.
\end{proof}

\subsection{Proof of \cref{main}}

The parent-child pairs correspond to the interior-face pairs of the horns to be filled.
We will fill these horns in lexicographically increasing order of their dimension, \emph{corank} and level.

\begin{definition}
	Let $x$ be an $m$-simplex in $\O(-,n) \setminus A[n]$.
	By the \emph{corank} of $x$, we mean
	\[
	\corank(x) \defeq \max\bigl\{|\alpha^{-1}(n)| : \alpha \in \supp(x)\bigr\}.
	\]
\end{definition}
\begin{remark}\label{rank corank}
	The corank of $x$ may be computed as
	\[
	\corank(x) = \dim(x) - \rank(x)+1.
	\]
	Using this observation, it is easy to deduce from \cref{dagchild,dagparent} that both the parent and child constructions preserve the corank.
\end{remark}

\begin{lemma}\label{dagfaces}
	Let $w$ be a non-degenerate $m$-simplex in $\O(-,n)\setminus A[n]$ with rank $r$, level $\ell$, and corank $c$.
	Suppose that $w$ satisfies $(\ddagger)$.
	Fix $k \in [m]$ with $k \neq \ell$.
	Then at least one of the following holds:
	\begin{itemize}
		\item $w \face k$ is degenerate;
		\item $w \face k$ is contained in $A[n]$;
		\item $w \face k$ is a non-degenerate simplex in $\O(-,n) \setminus A[n]$ satisfying $(\ddagger)$; or
		\item $w \face k$ is a non-degenerate simplex in $\O(-,n) \setminus A[n]$ not satisfying $(\ddagger)$, and we have
		\[
		\bigl(\corank(w\face k),\level(w \face k)\bigr) <_{\mathrm{lex}} (c,\ell).
		\]
	\end{itemize}
\end{lemma}
Here $<_{\mathrm{lex}}$ denotes the lexicographical order, so the assertion is that either we have $\corank(w\face k)<c$, or we have $\corank(w \face k) = c$ and $\level(w \face k)<\ell$.
\begin{proof}
	Assume that none of the first three disjuncts holds, that is, $w \face k$ is a non-degenerate simplex in $\O(-,n) \setminus A[n]$ not satisfying $(\ddagger)$.
	We would like to exhibit the inequality
	\[
	\bigl(\corank(w\face k),\level(w \face k)\bigr) <_{\mathrm{lex}} (c,\ell).
	\]
	Note that we have
	\[
	\supp(w\face k) \subset \{\alpha\face k : \alpha \in \supp(w)\}
	\]
	(but the two sides may not be equal).
	Since $|(\alpha\face k)^{-1}(n)| \le |\alpha^{-1}(n)|$ holds
	for any $\alpha : [m] \to [n]$, we can deduce $\corank(w \face k) \le c$.

	\begin{case}
		Suppose $k < \ell$.
		If $\corank(w \face k) < c$ then we are done, so assume $\corank(w \face k) = c$.
		We wish to show that $\level(w \face k) < \ell$.
		
		It follows from the relationship between rank and corank that $\rank(w \face k) = r-1$.
		Since we have $k < \ell \le r$, a simplicial operator $\alpha \in \supp(w)$ satisfies $\alpha(r) = n$ if and only if it satisfies $(\alpha\face k)(r-1) = n$.
		Hence we have $\widebar{\widebar{w \face k}} = \bbar w \face k$ and $\widecheck{w \face k} = \check w \face k$.
		It thus suffices to show that any $\alpha \in \supp(\check w)$ satisfies $(\alpha\face k)(\ell-1) = (\alpha \face k)(r-1)$, which indeed follows from $\alpha(\ell) = \alpha(r)$ and the inequality $k < \ell$.
	\end{case}
	\begin{case}
		Suppose $\ell < k < r$.
		Similarly to the previous case, we may assume $\corank(w \face k) = c$, and deduce that $\rank(w \face k) = r-1$, $\widebar{\widebar{w \face k}} = \bbar w \face k$ and $\widecheck{w \face k} = \check w \face k$.
		Since any $\alpha \in \supp(\check w)$ satisfies
		\[
		(\alpha\face k)(\ell) = \alpha(\ell) = \alpha(r) = (\alpha\face k)(r-1),
		\]
		we have $\level(w \face k) \le \ell$.
		We would like to show that this inequality is strict.
		
		Recall that $w$ satisfies $(\ddagger)$, that is, $\bbar w$ is degenerate at $\ell-1$.
		Since we have $\ell < k$, this implies that $\widebar{\widebar{w \face k}} = \bbar w \face k$ is also degenerate at $\ell-1$.
		So $\level(w \face k) = \ell$ would contradict our assumption that $w \face k$ does not satisfy $(\ddagger)$.
		Therefore we indeed have $\level(w \face k) < \ell$.
	\end{case}
	\begin{case}
		Suppose $k \ge r$.
		Observe that, in this case, if $\alpha \in \supp(w)$ satisfies $|\alpha^{-1}(n)| = c$ (or equivalently $\alpha(r)=n$) then $\alpha(k) = n$.
		Thus we have $|(\alpha\face k)^{-1}(n)| = c-1$ for any such $\alpha$.
		It then follows from $\supp(w\face k) \subset \{\alpha\face k : \alpha \in \supp(w)\}$ that $\corank(w \face k) < c$.
	\end{case}
	This completes the proof.
\end{proof}

\begin{lemma}\label{dagmarked}
	Let $x$ be a non-degenerate simplex in $\O(-,n) \setminus A[n]$ with level $\ell$.
	Suppose that $x$ is marked in $\O(-,n)$ and does not satisfy $(\ddagger)$, and let $w$ be its parent.
	Then the faces $w \face{\ell-1}$ and $w \face{\ell+1}$ are also marked in $\O(-,n)$.
\end{lemma}
\begin{proof}
	By \cref{marked}, the assumption that $x$ is marked is equivalent to every operator in $\supp(x)$ being degenerate.
	Recall from the proof of \cref{dagparent} that we have
	\[
	w \face{\ell+1} = (u \filler \ell v)\face{\ell+1} = u
	\]
	where $u = \check x + \bbar x \face \ell \degen{\ell-1}$ and $v = \check x \face{\ell-1}\degen{\ell-1}+\bbar x$.
	Since we have
	\[
	\supp(u) \subset \supp(\check x) \cup \{\alpha\face \ell \degen{\ell-1} : \alpha \in \supp(\bbar x)\}
	\]
	and every simplicial operator in the larger set is degenerate, $u$ is marked by \cref{marked}.
	Similarly $w\face{\ell-1}=v$ is marked.
\end{proof}

Now we are ready to prove our main theorem.

\begin{proof}[Proof of \cref{main}]
We proceed by induction on $n$.
The base case is trivial since $\iota_0 : \Delta[0] \to \O(-,0)$ is invertible.

For the inductive step, fix $n \ge 0$ and suppose that $\iota_{n-1}$ is an inner complicial anodyne extension.
Consider the factorisation
\[
\Delta[n] \to A[n] \incl \O(-,n)
\]
of the map $\iota_n$.
The first factor $\Delta[n] \to A[n]$ is isomorphic to the join of $\iota_{n-1}$ and the identity on $\Delta[0]$ by \cref{iota join}.
Hence this factor is an inner complicial anodyne extension by the inductive hypothesis and \cref{join}.

It remains to prove that the inclusion $A[n] \incl \O(-,n)$ is an inner complicial anodyne extension.
Let $I$ be the set
\[
I = \bigl\{(m,c,\ell) \in \NN^3 : c,\ell \le m\bigr\}
\]
equipped with the lexicographical ordering (so that it has order type $\omega$).
For each $(m,c,\ell) \in I$, let:
\begin{itemize}
	\item $P^{m,c,\ell}$ be the set of parent simplices of dimension $m$, corank $c$, and level $\ell$;
	\item $Q^{m,c,\ell} \subset P^{m,c,\ell}$ be the subset consisting of those $x$ such that $x \face \ell$ is marked in $\O(-,n)$; and
	\item $B^{m,c,\ell} \subset \O(-,n)$ be the smallest regular subset containing $A[n]$ and $P^{m',c',\ell'}$ for all $(m',c',\ell') \le_{\mathrm{lex}} (m,c,\ell)$.
\end{itemize}
Then the assignation $(m,c,\ell) \mapsto B^{m,c,\ell}$ defines a transfinite sequence in $\msSet$ whose composite is precisely the inclusion $A[n] \incl \O(-,n)$.
Hence it suffices to show that, for each immediate predecessor-successor pair $(m',c',\ell') <\lex (m,c,\ell)$, the inclusion $B^{m',c',\ell'} \incl B^{m,c,\ell}$ is an inner complicial anodyne extension.
The following claim will be useful.

\setcounter{theorem}{1}
\begin{claim}\label{missing simplices}
	The set of non-degenerate simplices in $B^{m,c,\ell} \setminus B^{m',c',\ell'}$ is precisely $\{w,w\face \ell : w \in P^{m,c,\ell}\}$.
\end{claim}
\begin{proof}[Proof of Claim]
	This follows from \cref{dagpairing,dagfaces}.
	Note that we have either $m' = m$ or $(m',c',\ell') = (m-1,m-1,m-1)$.
	This implies that $B^{m',c',\ell'}$ contains all $(m-2)$-simplices of $\O(-,n)$.
	Thus, if $w \in P^{m,c,\ell}$ has a degenerate face $w \face k$ (with $k \neq \ell$), this face must be contained in $B^{m',c',\ell'}$.
\end{proof}

Observe that, since each $w \in P^{m,c,\ell}$ may be written as $w = u \filler \ell v$, it induces a map $\asimp m \ell \to B^{m,c,\ell}$ by \cref{filler complicial}, which moreover restricts to $\horn m \ell \to B^{m',c',\ell'}$ by \cref{missing simplices}.
Consider the upper left pushout square in the following diagram.
\[
\begin{tikzcd}[row sep = large]
\ds\coprod_{w \in P^{m,c,\ell}}\horn m \ell
\arrow [r] 
\arrow [d]
\arrow [dr, phantom, "\ulcorner" very near end] &
\ds\coprod_{w \in P^{m,c,\ell}}\asimp m \ell
\arrow [d] & \\
B^{m',c',\ell'}
\arrow [r] &
C^{m,c,\ell}
\arrow [r] &
B^{m,c,\ell}\\
& \ds\coprod_{w \in Q^{m,c,\ell}} \asimpd m \ell
\arrow [u]
\arrow [r]
\arrow [ur, phantom, "\llcorner" very near end] &
\ds\coprod_{w \in Q^{m,c,\ell}} \asimpdd m \ell
\arrow [u]
\end{tikzcd}
\]
Again by \cref{missing simplices}, we may describe the pushout object $C^{m,c,\ell}$ as precisely $B^{m,c,\ell}$ except that all simplices of the form $w \face \ell$ with $w \in P^{m,c,\ell}$ are unmarked.
It then follows from \cref{dagmarked} that the inclusion $C^{m,c,\ell} \incl B^{m,c,\ell}$ may be written as a pushout as indicated above.
This completes the proof.
\end{proof}

\setcounter{theorem}{23}

\begin{remark}
	For $n \ge 1$, let $\Delta[n]_t$ denote the \emph{standard marked $n$-simplex} (obtained from $\Delta[n]$ by marking the unique non-degenerate $n$-simplex).
	Since
	\[
	\{\Delta[n] : n \ge 0\} \cup \{\Delta[n]_t : n \ge 1\}
	\]
	spans the ``standard'' dense subcategory of $\msSet$ (\emph{cf.}\;\cite[\textsection 5]{Street:oriented}, \cite[Observation 12]{Verity:I}, or \cite{OR:precomplicial}), one might be interested in proving an analogue of \cref{main} for those marked simplices.
	More precisely, one might consider $\O_n$ with the atomic $n$-cell collapsed to an identity (or inverted), and hope to prove its complicial nerve to be weakly equivalent to $\Delta[n]_t$.
	Unfortunately, such a result seems to require much more work than what we have developed here, mainly because these modified orientals cannot be treated using Steiner's machinery.
	The ``most fibrant'' object we can currently prove to be weakly equivalent to $\Delta[n]_t$ is $N\O_n$ with all $n$-simplices additionally marked (which is not actually fibrant), though we will omit the proof.
\end{remark}

\section*{Acknowledgements}
The author gratefully acknowledges the support of an Australian Mathematical Society Lift-Off Fellowship.
During the revision of this work, he was also supported by JSPS KAKENHI Grant Number JP21K20329.

The author would like to thank the anonymous referee whose careful reading and constructive comments greatly improved the rigour and the readability of this paper.

\bibliographystyle{alpha}
\bibliography{ref}

\end{document}